\documentclass{article}
\usepackage{graphicx} % Required for inserting images
\usepackage{graphicx} % Required for inserting images
\usepackage[french,english]{babel}
\usepackage{amsmath}
\usepackage{framed}
\usepackage[utf8]{inputenc}
\usepackage{amsmath,amsfonts,amssymb}
\usepackage{pifont}
\usepackage{hyperref}
\usepackage{csquotes}
\usepackage{empheq}
\usepackage{tcolorbox}
\usepackage{xcolor} % for custom colors
\graphicspath{{./images}}
\usepackage[top=2.5cm, bottom=2.5cm, left=2.5cm, right=2.5cm]{geometry}
\usepackage{hyperref}
\usepackage{fancybox}
\usepackage{tikz}
\usepackage{ragged2e}
\usepackage{lmodern}
\usepackage{xcolor}
\usepackage{quoting}
\usepackage{ragged2e}
\usepackage{setspace}
\usepackage{titlesec}
\usepackage[T1]{fontenc}
\usepackage{mathtools}
\usepackage{array}
\usepackage{amsfonts,amsmath,amssymb}
\usepackage{pifont}
\usepackage{graphicx}
\usepackage{bbold}
\usepackage[svgnames]{xcolor}
\usepackage{comment}
\usepackage{fancyhdr}
\usepackage{fancybox}
\usepackage{caption}
\usepackage{appendix}
\usepackage{amsthm} 
\usepackage{longtable}

\usepackage{mathrsfs}

\newtheorem{theorem}{Theorem}[section]

\newtheorem{proposition}[theorem]{Proposition}
\newtheorem{lemma}[theorem]{Lemma}
\newtheorem{corollary}[theorem]{Corollary}

\theoremstyle{definition}
\newtheorem{definition}[theorem]{Definition}

\theoremstyle{remark}
\newtheorem{remark}[theorem]{Remark}

\newcommand{\R}{\mathbb{R}}
\newcommand{\dd}{\,\mathrm{d}}
\newcommand{\PP}{\mathbb{P}}
\newcommand{\PV}{PV}
\newcommand{\QQ}{\mathbb{Q}}
\newcommand{\EE}{\mathbb{E}}
\newcommand{\Law}{\mathcal{L}}
\newcommand{\spt}{\operatorname{supp}}

\newcommand{\Xh}{X_{\mathrm{hyb}}}

\begin{document}
%=================================================================
\begin{center}
{\LARGE\bfseries
New link between the fractional p-Laplacian operators and a class of McKean-Vlasov flight type processes\\[6pt]
}

\vspace{1.5em}

{\large
Houssine~EL~JEDDAOUI\, \quad Dany~NABAB\,
}

\vspace{0.8em}

\vspace{0.8em}

{\small \today}
\end{center}

\vspace{1em}

%\tableofcontents
\noindent
\vspace{1em}
 
%=================================================================
\noindent\textbf{Abstract.}
We prove the existence of a McKean--Vlasov stochastic process with jumps associated to
the nonlinear parabolic equation
$\partial_t u = \Delta_p u + \Delta_p^s u$
in $\R^N\times(0,\infty)$, where $\Delta_p$ is the $p$-Laplacian and $\Delta_p^s$ is the fractional
$p$-Laplacian.
Our approach is the following: first, after proving the existence of a solution for the PDE presented earlier, we rewrite it as a nonlinear Fokker-Planck-Kolmogorov equation whose measure-valued solution is guaranteed when $p\ge4$. Then we solve the martingale problem associated to our FPKE via a new superposition principle for non-local Fokker-Planck-Kolmogorov equations, which only requires an integrability condition on the local coefficients. Finally, thanks to the martingale solution obtained, we derive the existence of a weak solution for the McKean-Vlasov's type SDE with jumps whose infinitesimal generator is a ``hybrid version'' of the operator $\Delta_p+\Delta_p^s$.
\vspace{0.5em}
 
 
\vspace{1em}
 
 
\vspace{1em}
 
\section{Introduction}

\noindent
One of the most fruitful ideas in twentieth-century mathematics
is the correspondence between differential operators, stochastic processes and parabolic Partial Differential Equations. By correspondence, we mean that an abstract operator $A$ may be the infinitesimal generator of a stochastic process $X_t$, in the sense that
\[
Af(x) = \lim_{h\to 0}\frac{\mathbb{E}[f(X_{t+h})|X_t=x]-f(x)}{h},
\]
and that the law $\mu_t$ of $X_t$ then solves the forward Kolmogorov (or Fokker--Planck) equation $\partial_t\mu_t=A^*\mu_t$.

\medskip
\noindent
This correspondence may be traced back to Bachelier \cite{Bachelier1900,Kahane1998} in 1900, Einstein \cite{Einstein1905} in 1905 and Smoluchowski \cite{Smoluchowski1906} in 1906, who related Brownian motion to the heat equation $\partial_tu=\Delta u$ (see also Brown \cite{Brown1828}, Perrin \cite{Perrin1909} and Wiener \cite{Wiener1923}). In 1931, Kolmogorov \cite{Kolmogorov1931} developed the analytic theory of continuous-time Markov processes and derived the backward and forward equations, so that the correspondence between the heat equation and Brownian motion became a correspondence between Kolmogorov equations and Markov processes. The works of de Finetti \cite{DeFinetti1929}, L\'evy \cite{Levy1934}, Khintchine \cite{Khintchine1937} and It\^o \cite{Ito1951} on infinitely divisible laws, the L\'evy--Khintchine formula and the L\'evy--It\^o decomposition then showed that the noise driving such processes may contain a drift, a Brownian component and jumps, Hunt \cite{Hunt1956} characterized the generators of L\'evy processes, and Gikhman and Skorokhod \cite{GikhmanSkorokhod1972} developed the theory of stochastic differential equations with jumps; see \cite{Sato1999,Applebaum2009,Schilling2016} for modern presentations. Accordingly, the generator of a (time-inhomogeneous) Markov jump-diffusion has the L\'evy-type form
\begin{equation}
\mathcal{L}_tf(x)=b_t(x)\cdot\nabla f(x)+\mathrm{tr}\bigl(a_t(x)\nabla^2f(x)\bigr)+\int_{\R^N\setminus\{0\}}\bigl(f(x+z)-f(x)-\mathbf{1}_{\{|z|\le1\}}z\cdot\nabla f(x)\bigr)\,\nu_{t,x}(\dd z),
\label{eq:Levy_type_intro}
\end{equation}
where $a_t(x)$ is a nonnegative definite matrix, $b_t(x)$ a drift and $\nu_{t,x}$ a family of L\'evy measures.

\medskip
\noindent
In all these works, the dynamics of the process do not depend on its own law. The idea of a process whose coefficients depend on its distribution goes back to Vlasov \cite{Vlasov1938}, who described in 1938 the mean-field limit of a large system of charged particles by a nonlinear transport equation, and to Kac \cite{Kac1956} and McKean \cite{McKean1966,McKean1967,McKean1975}, who related interacting particle systems with random collisions to the Boltzmann equation. Tanaka \cite{Tanaka1978} gave in 1978 a probabilistic representation of the Boltzmann equation by a stochastic differential equation driven by jumps whose intensity depends on the law of the solution, and Graham \cite{Graham1992} studied in 1992 McKean--Vlasov equations with non-compensated jumps of the form
\begin{equation*}
\dd X_t = b\bigl(X_t,\mathcal{L}(X_t)\bigr)\,\dd t + \sigma\bigl(X_t,\mathcal{L}(X_t)\bigr)\,\dd B_t + \int_E \gamma\bigl(X_{t^-},\mathcal{L}(X_t),z\bigr) \,\mathcal{N}(\dd t,\dd z),
\end{equation*}
where $\mathcal{N}$ is a Poisson random measure, proving existence, uniqueness and propagation of chaos. We refer to \cite{Funaki1984,Sznitman1984,Sznitman1991} for the mathematical theory of McKean--Vlasov equations, and to \cite{Carmona2015,Benazzoli2020,CormierTanreVeltz2020,Yoshioka2025} for some applications.

\medskip
\noindent
In 2024, Barbu, Rehmeier and R\"ockner \cite{Barbu2024} found a McKean--Vlasov interpretation of the $p$-Laplacian. Rewriting $\partial_t u = \Delta_p u$ as a nonlinear Fokker--Planck equation and applying the superposition principle of Trevisan \cite{Trevisan2016}, they proved the existence of the so-called \emph{$p$-Brownian motion}
\begin{equation*}
\dd X_t
= \nabla(|\nabla u|^{p-2})(t,X_t)\,\dd t
+\sqrt{2}\,|\nabla u(t,X_t)|^{\frac{p-2}{2}}\,\dd W_t,
\end{equation*}
whose generator is $Af= \nabla(|\nabla u|^{p-2})\cdot\nabla f+|\nabla u|^{p-2}\,\Delta f$. A similar method has been used by Barbu and R\"ockner \cite{BarbuRockner2024} for a nonlinear Fokker--Planck equation with fractional Laplacian and McKean--Vlasov equations with L\'evy noise.

\medskip
\noindent
The natural question is then: what is the process associated to the operator $A=\Delta_p+\Delta_p^s$, composed of the $p$-Laplacian $\Delta_p$ and the fractional $p$-Laplacian
\begin{equation*}
\Delta_p^s u(x)
= \mathrm{P.V.}\!\int_{\R^N}
\frac{|u(y)-u(x)|^{p-2}(u(y)-u(x))}
{|x-y|^{N+sp}}\,\dd y,
\quad s\in(0,1)?
\end{equation*}
Mixed local and nonlocal operators have attracted much attention in recent years from the analytic point of view: see for instance Biagi, Dipierro, Valdinoci and Vecchi \cite{Biagi2022} in the linear case, and Garain and Kinnunen \cite{GarainKinnunen2022,GarainKinnunen2024} for the elliptic and parabolic equations driven by $\Delta_p+\Delta_p^s$. In the linear case $p=2$, the operator $\Delta+\Delta^{s}_2$ is, up to constants, the generator of the sum of a Brownian motion and an independent isotropic $2s$-stable process \cite{ChenKimSongVondracek2012}. To the best of our knowledge, the stochastic process associated with the nonlinear operator $\Delta_p+\Delta_p^s$ has not been studied. In this paper, we answer this question: the associated process extends the $p$-Brownian motion of \cite{Barbu2024} by adding jumps whose intensity depends on the law of the process itself.

\subsection*{Main results}

\noindent
Let $u$ be the solution of
\begin{equation}
\partial_t u=\Delta_p u+\Delta_p^s u\quad\mbox{in }(0,\infty)\times\R^N,\qquad u(0)=u_0,
\label{eq:PDE_intro}
\end{equation}
constructed in Theorem~\ref{thm:existence_PDE} by the nonlinear Hille--Yosida theorem of Br\'ezis \cite{Brezis1973}, and set
\[
\nu_{t,x}(\dd z):=\frac{|u(t,x)-u(t,x+z)|^{p-2}}{|z|^{N+sp}}\,\dd z .
\]
Our main result is the following.

\begin{theorem}[McKean--Vlasov flight process associated to $\Delta_p+\Delta_p^s$]
\label{thm:main_intro}
Let $p\ge4$, $s\in(0,1)$ with $1<sp<2$ and $sp<N\le p$, and let $u_0$ be a probability density satisfying assumptions \textup{(A2)--(A6)} of Section~\ref{sec:PDE}. Then there exist a stochastic basis carrying a Brownian motion $W$, an integer-valued random measure $\mathcal{N}$ with compensator $\nu_{t,X_{t^-}}(\dd z)\,\dd t$, and a c\`adl\`ag process $X$ such that
\begin{equation*}
\left\{
\begin{aligned}
\dd X_t
&= \nabla(|\nabla u|^{p-2})(t,X_t)\,\dd t
+\sqrt{2}\,|\nabla u(t,X_t)|^{\frac{p-2}{2}}\,\dd W_t
+\int_{\R^N\setminus\{0\}}z\,\mathcal{N}(\dd t,\dd z),\\
\mathrm{Law}(X_t)&=u(t,x)\,\dd x,\quad t\ge0.
\end{aligned}
\right.
\end{equation*}
In particular, $\Delta_p+\Delta_p^s$ is the infinitesimal generator, in the probabilistic sense, of the nonlinear (McKean--Vlasov) process $X$: rewriting \eqref{eq:PDE_intro} as a nonlinear Fokker--Planck equation, its generator takes the L\'evy-type form \eqref{eq:Levy_type_intro} with $a_t=|\nabla u|^{p-2}I$ and $\nu_{t,x}$ as above, and $X$ solves an equation of Graham type whose coefficients depend on the law of $X_t$ through its density $u(t)$.
\end{theorem}

\noindent
Theorem~\ref{thm:main_intro} is proved in Theorem~\ref{thm:5.4}. Let us comment on the assumptions. The condition $1<sp<2$ ensures that $\nu_{t,x}$ is a L\'evy measure whose small jumps have a finite first moment (Lemma~\ref{lem:compensator}), the condition $p\ge4$ is used in the integrability estimates of Sections~\ref{sec:PDE} and~\ref{sec:MP}, and $N\le p$ is only used to prove the conservation of mass. The compact support assumption (A5) on $u_0$ is only used to ensure that $\nabla u_0\in L^2(\R^N)$, and could be replaced by this condition. We do not address the uniqueness in law of $X$, which is related to the uniqueness of measure-valued solutions of the associated nonlinear Fokker--Planck--Kolmogorov equation; this question is left for future work.

\medskip
\noindent
The proof of Theorem~\ref{thm:main_intro} relies on three results which are of independent interest.
\begin{itemize}
    \setlength{\itemsep}{1.5ex}
    \item[$\bullet$] \textit{A global weighted second-order estimate} (Corollary~\ref{cor:global_DiBe}): the solution of \eqref{eq:PDE_intro} satisfies
    \[
    \int_0^T\!\!\int_{\R^N}|\nabla u|^{p-2}|D^2u|^2\,\dd x\,\dd t\le\gamma\,\|\nabla u_0\|_{L^2(\R^N)}^2,
    \]
    which extends the local estimate of DiBenedetto \cite{DiBenedetto1993} to the whole space and up to $t=0$, without any assumption on the support of $u(t)$.
    \item[$\bullet$] \textit{The integrability of the compensator} (Lemma~\ref{lem:compensator}): although $sp>1$, the small jumps of $X$ have a finite first moment, since $u(t)\in W^{1,p}(\R^N)$.
    \item[$\bullet$] \textit{A superposition principle for non-local Fokker--Planck--Kolmogorov equations under integrability conditions} (Theorem~\ref{thm:15prime_weak}), which we now describe.
\end{itemize}

\medskip
\noindent
The idea of lifting a solution of a Fokker--Planck equation to a solution of the associated martingale problem goes back to the forward-equation results of Ethier and Kurtz \cite{EthierKurtz1986}. In 2001, Kurtz and Stockbridge \cite{KurtzStockbridge2001} proved, for very general operators (including L\'evy-type operators with measurable coefficients, see \cite[Remark~1.10]{KurtzStockbridge2001}), that every solution of the forward equation satisfying an integrability condition is the family of one-dimensional distributions of a solution of the martingale problem. However, as they point out themselves \cite[Remark~3.5]{KurtzStockbridge2001}, this process need not have c\`adl\`ag paths in $\R^d$: it may ``go out'' at infinity and ``come back in''. For diffusion operators, Figalli \cite{Figalli2008} established in 2008 a superposition principle on the path space for bounded coefficients, and Trevisan \cite{Trevisan2016} extended it in 2016 to coefficients satisfying only an integrability condition of the type
\[
\int_0^T\!\!\int_{\R^d}\Bigl(\frac{|a_t(x)|}{1+|x|^2}+\frac{|b_t(x)|}{1+|x|}\Bigr)\mu_t(\dd x)\,\dd t<\infty.
\]
In the non-local case, Fournier and Xu \cite{FournierXu2016} obtained the equivalence between jumping SDEs and non-local PDEs for coefficients with a pointwise linear growth and jumps of finite variation, and R\"ockner, Xie and Zhang \cite{RXZ2020} proved in 2020 a superposition principle for general non-local Fokker--Planck--Kolmogorov operators under the uniform growth condition
\[
\sup_{t,x}\Bigl(\frac{|a_t(x)|+g_t^\nu(x)}{1+|x|^2}+\frac{|b_t(x)|}{1+|x|}+h_t^\nu(x)\Bigr)<\infty,
\]
\noindent where the small and big jumps are respectively controlled by:
\begin{itemize}
    \setlength{\itemsep}{2ex}
    \item[$\bullet$] $g_t^\nu(x) := \displaystyle\int_{|z|\le\ell} |z|^2 \, \nu_{t,x}(\dd z)$,
    \item[$\bullet$] $h_t^\nu(x) := \displaystyle\int_{|z|>\ell} \log\biggl(1+\frac{|z|}{1+|x|}\biggr) \, \nu_{t,x}(\dd z)$.
\end{itemize}

\medskip

\noindent They observed in \cite[Remark~1.7]{RXZ2020} that the push-forward method of Trevisan does not seem to apply in the non-local case, leaving an integrated version of their result open. Our second main result partially answers this question, the integrability condition being imposed on the local coefficients.

\begin{theorem}[Superposition principle under integrability conditions]
\label{thm:superposition_intro}
Let $(\mu_t)_{t\ge0}$ be a weak solution of probability measures of the non-local Fokker--Planck--Kolmogorov equation $\partial_t\mu_t=\mathcal{L}_t^*\mu_t$, where $\mathcal{L}_t$ is a L\'evy-type operator with measurable coefficients $(a,b,\nu)$. Assume that $\mu_t$ has a density, that the local coefficients satisfy the integrability condition
\[
\int_{t_1}^{t_2}\!\!\int_{\R^N}\Bigl(\frac{|a_t(x)|}{1+|x|^2}+\frac{|b_t(x)|}{1+|x|}\Bigr)\mu_t(\dd x)\,\dd t\le C_T(t_2-t_1)^\alpha,\qquad 0\le t_1<t_2\le T,
\]
and that the jump part satisfies $\sup_{t,x}\bigl(g_t^\nu(x)/(1+|x|^2)+h_t^\nu(x)\bigr)<\infty$. Then there exists a martingale solution $\QQ$ on the path space $\mathbb{D}([0,\infty);\R^N)$ such that $\QQ\circ X_t^{-1}=\mu_t$ for all $t\ge0$.
\end{theorem}

\noindent
Theorem~\ref{thm:superposition_intro} is stated precisely in Theorem~\ref{thm:15prime_weak}. Its proof relies on three main components:
\begin{itemize}
    \setlength{\itemsep}{1.5ex}
    \item[$\bullet$] The approximation scheme of \cite{RXZ2020}.
    \item[$\bullet$] The result of Kurtz and Stockbridge applied at the regularized level (Lemma~\ref{lem:Qeps}).
    \item[$\bullet$] A non-explosion argument based on the integrability assumption.
\end{itemize}

\smallskip

\noindent In our setting, the uniform condition of \cite{RXZ2020} cannot be used, since the coefficients $|\nabla u|^{p-2}$ and $\nabla(|\nabla u|^{p-2})$ are not known to be bounded. Instead, the required integrability assumption is obtained from the global weighted second-order estimate of Corollary~\ref{cor:global_DiBe}.

\medskip
\noindent
\textbf{Organization of the paper.} In Section~\ref{sec:notation}, we fix the notation and prove that $\nu_{t,x}$ is a L\'evy measure. In Section~\ref{sec:PDE}, we prove the existence of a weak solution of \eqref{eq:PDE_intro}, together with positivity, boundedness and mass conservation, by the theory of maximal monotone operators of Br\'ezis \cite{Brezis1973}; we then prove the integrability of the compensator and rewrite \eqref{eq:PDE_intro} as a nonlinear Fokker--Planck--Kolmogorov equation. In Section~\ref{sec:MP}, we prove the superposition principle and deduce the existence of a solution of the associated martingale problem. In Section~\ref{sec:SDE}, we construct the McKean--Vlasov process of Theorem~\ref{thm:main_intro}, using the results of Jacod and Shiryaev \cite{JS1987}. Finally, the weighted second-order estimates are proved in the Appendix.
%=================================================================
\section{Notation and function spaces}
\label{sec:notation}
%=================================================================

\noindent
Throughout this paper, $N\geq 1$, $p\geq 4$,
$s\in(0,1)$ with $1<sp<2$,
$T>0$ is a finite time,
and $\mathcal{O}\subset\R^N$ is an open set.
We write $\mathcal{O}_T:=\mathcal{O}\times(0,T)$.
\subsection{Analytic framework}
\noindent
\textbf{The operators.} The \emph{$p$-Laplacian}
\begin{equation}
\Delta_p u
:= \operatorname{div}(|\nabla u|^{p-2}\nabla u),\quad p>1,
\label{eq:plap}
\end{equation}
has been studied extensively since the works of
DiBenedetto \cite{DiBenedetto1993} and Lions \cite{Lions1969}.
The \emph{fractional $p$-Laplacian}
\begin{equation}
\Delta_p^s u(x)
:= \mathrm{P.V.}\int_{\R^N}
\frac{|u(y)-u(x)|^{p-2}(u(y)-u(x))}
{|x-y|^{N+sp}}\,\dd y,
\quad x\in\R^N,
\label{eq:fracplap}
\end{equation}
is a nonlocal operator of order $2s$ introduced by
Di Nezza, Palatucci and Valdinoci \cite{DiNezza2012} (see also Del Pezzo and Quaas \cite{DelPezzo2016}) and studied in the
evolution setting by Maz\'on, Rossi and Toledo \cite{Mazon2016}; see also \cite{Lindgren2014,DiCastro2016,Stroemqvist2019,Vazquez2019} for regularity and qualitative properties, and \cite{Vazquez2006} for nonlinear diffusion equations.\\

\noindent
\textbf{Function spaces.} We use the standard Lebesgue and Sobolev spaces $L^q(\R^N)$, $W^{1,p}(\R^N)$, their local versions $L^q_{\mathrm{loc}}$, $W^{1,p}_{\mathrm{loc}}(\mathcal{O})$, and the fractional Sobolev space
\[
W^{s,p}(\R^N)
:= \left\{u\in L^p(\R^N)
\;\middle|\;
[u]_{W^{s,p}(\R^N)}^p
:= \iint_{\R^N\times\R^N}
\frac{|u(x)-u(y)|^p}{|x-y|^{N+sp}}
\,\dd x\,\dd y<\infty\right\},
\]
endowed with the norm $\|u\|_{W^{s,p}}:= \|u\|_{L^p}+[u]_{W^{s,p}}$, see \cite{Brezis2011,DiNezza2012,Lions1972}. We recall the Sobolev embedding $W^{s,p}(\R^N)\hookrightarrow L^{p_s^*}(\R^N)$, $p_s^* = \frac{Np}{N-sp}$, for $sp<N$ \cite{DiNezza2012}. For a Banach space $X$, $L^q(0,T;X)$ denotes the Bochner space of strongly measurable functions $u:(0,T)\to X$ with $\|u(\cdot)\|_X\in L^q(0,T)$, $C([0,T];X)$ the space of continuous functions from $[0,T]$ into $X$, $W^{1,\infty}(0,T;X)$ the space of functions $u\in L^\infty(0,T;X)$ with $\frac{du}{dt}\in L^\infty(0,T;X)$, and $L^q_{\mathrm{loc}}(0,\infty;X)$ the space of functions belonging to $L^q(0,T;X)$ for every $T>0$. Finally,
\[
C_{\mathrm{loc}}(0,T;L^2_{\mathrm{loc}}(\mathcal{O}))
:= \bigl\{u \mid
u\in C([t_1,t_2];L^2(K))
\;\forall\,[t_1,t_2]\subset(0,T),\,
\forall\,K\Subset\mathcal{O}\bigr\}.
\]
We set $H:=L^2(\R^N)$, with inner product $\langle u,v\rangle_H = \int_{\R^N}uv\,\dd x$, and
\begin{equation}\label{eq:V}
V := W^{1,p}(\R^N)\cap W^{s,p}(\R^N)\cap L^2(\R^N),
\quad
\|u\|_V := \|u\|_{L^2}+\|\nabla u\|_{L^p}+[u]_{W^{s,p}}.
\end{equation}
$V$ is a reflexive Banach space, and $V \hookrightarrow H \hookrightarrow V^{*}$ is a Gelfand triple, where the duality pairing between $V^*$ and $V$ is denoted by ${}_{V^*}\langle \cdot,\cdot\rangle_V$ \cite{Brezis2011}.\\

\noindent
\textbf{Monotone operators.} We use the terminology of \cite{Barbu2010,Brezis1973}. We consider $A := -\Delta_p - \Delta_p^s$ with domain $D(A) := \{u\in V : -\Delta_p u-\Delta_p^s u\in L^2(\mathbb{R}^N)\}$, which contains $D(-\Delta_p) \cap D(-\Delta_p^s)$. An operator $A : D(A) \subset H \to H$ is \emph{monotone} if $\langle Au - Av, u - v \rangle_H \geq 0$ for all $u, v \in D(A)$, and \emph{maximal monotone} if, in addition, $R(I + A) = H$ (Minty's theorem). An operator $A_V : V \to V^{*}$ is \emph{hemicontinuous} if $t \mapsto {}_{V^{*}}\langle A_V(u + tv), w \rangle_V$ is continuous at $t = 0$ for all $u, v, w \in V$. For a convex, lower semicontinuous functional $\Phi : H \to (-\infty, +\infty]$, the subdifferential $\partial\Phi(u) := \{w \in H : \Phi(v) \geq \Phi(u) + \langle w, v - u \rangle_H, \; \forall \, v \in H\}$ is maximal monotone.

\subsection{Probabilistic framework}
\noindent
For a random variable $X$, $\mathcal{L}(X)$ or $\mathrm{Law}(X)$ denotes its law, and $\mathcal{P}(E)$ is the set of probability measures on a Polish space $E$. We denote by $\mathbb{D}:=\mathbb{D}([0,\infty);\R^N)$ the space of c\`adl\`ag paths, endowed with the Skorokhod topology, by $X_t(\omega):=\omega(t)$ the canonical process and by $(\mathcal{F}_t)_{t\ge0}$ its natural filtration. For integer-valued random measures, their compensators and the characteristics of semimartingales, we follow Jacod and Shiryaev \cite{JS1987}.\\

\noindent\textbf{L\'evy measure.} A $\sigma$-finite Borel measure $\nu$ on $\R^N\setminus\{0\}$
is a \emph{L\'evy measure} if
\begin{equation}
\int_{\R^N\setminus\{0\}}(1\wedge|y|^2)\,\nu(\dd y) < \infty.
\label{eq:levy_def}
\end{equation}
\subsection{Tool results}
\noindent In what follows, we consider the particular Borel measure
\begin{equation}
\nu_u(\dd y)
:= \frac{|u(x)-u(x+y)|^{p-2}}{|y|^{N+sp}}\,\dd y,
\qquad y\in\R^N\setminus\{0\},
\label{eq:nu_u}
\end{equation}
where $u$ is a measurable function.
\begin{lemma}[$\nu_u$ is a L\'evy measure]
\label{thm:levy}
If $p\ge2$, $sp<2$ and $u\in L^{\infty}(\R^N)$, the measure $\nu_u$ defined in
\eqref{eq:nu_u} is a Lévy measure.
More precisely:
\begin{equation}
\int_{\R^N\setminus\{0\}}(1\wedge|y|^2)\,\nu_u(\dd y)
\leq
\frac{2(2M)^{p-2}\omega_N}{sp(2-sp)},
\label{eq:levy_bound}
\end{equation}
 where $M:=\|u\|_{L^\infty}$ and $\omega_N := |\mathbb{S}^{N-1}|
= \frac{2\pi^{N/2}}{\Gamma(N/2)}$ is the surface area of the unit sphere $\mathbb{S}^{N-1}$
in $\R^N$, with $\displaystyle\Gamma(t):=\int_0^\infty s^{t-1}e^{-s}\,\dd s$
the Euler Gamma function
\end{lemma}

\begin{proof}
Since $|u(x)-u(x+y)|\le2M$ and $sp<2$, we get
\begin{equation*}
\begin{aligned}
\int_{\R^N\setminus\{0\}}(1\wedge|y|^2)\,\nu_u(\dd y)
&= \int_{|y|\leq 1}|y|^2\,\nu_u(\dd y)
+ \int_{|y|>1}\nu_u(\dd y)\\
&=\int_{|y|\leq 1}|y|^2\,
\frac{|u(x)-u(x+y)|^{p-2}}{|y|^{N+sp}}\dd y+\int_{|y|>1}\frac{|u(x)-u(x+y)|^{p-2}}{|y|^{N+sp}}\dd y\\
&\le(2M)^{p-2}
\int_{|y|\leq 1}\frac{|y|^2}{|y|^{N+sp}}\dd y+(2M)^{p-2}\int_{|y|>1}\frac{\dd y}{|y|^{N+sp}}\\
&= (2M)^{p-2}
\int_{|y|\leq 1}|y|^{2-N-sp}\dd y+(2M)^{p-2}\int_{|y|>1}\frac{\dd y}{|y|^{N+sp}}. 
\end{aligned}
\end{equation*}
Therefore, passing to spherical coordinates $y=r\omega$,
$r=|y|\in(0,1]$, $\omega\in\mathbb{S}^{N-1}$,
$\dd y=r^{N-1}\dd r\,\dd\omega$:
\begin{equation*}
\begin{aligned}
\int_{\R^N\setminus\{0\}}(1\wedge|y|^2)\,\nu_u(\dd y)
&\le(2M)^{p-2}\left(\int_{\mathbb{S}^{N-1}}\int_0^1 r^{2-N-sp}\cdot r^{N-1}\dd r\dd\omega\right.\\
&\qquad\left.+ \int_{\mathbb{S}^{N-1}}\int_1^\infty r^{-(N+sp)}\cdot r^{N-1}\dd r\dd\omega\right)\\
&= (2M)^{p-2}\left(\omega_N\int_0^1 r^{1-sp}\dd r + \omega_N\int_1^\infty r^{-sp-1}\dd r\right)\\
&= (2M)^{p-2}\left(\frac{\omega_N}{2-sp} +\frac{\omega_N}{sp}\right)\\
&=\frac{2(2M)^{p-2}\omega_N}{sp(2-sp)}, 
\end{aligned}
\end{equation*}
where $\omega_N=|\mathbb{S}^{N-1}|$. This completes the proof.
\end{proof} 
\section{From PDE \eqref{eq:PDE} to Fokker-Planck-Kolmogorov Equation \eqref{eq:FPE_general}}
\label{sec:PDE}

In what follows, we work under the following assumptions:
\begin{enumerate}
\item[(A1)] $p\ge2$, $s\in(0,1)$, $N\geq 1$.
\item[(A2)] $u_0\in L^1(\R^N)\cap L^2(\R^N)\cap L^\infty(\R^N)$.
\item[(A3)] $u_0\geq 0$ a.e.\ on $\R^N$.
\item[(A4)] $\displaystyle\int_{\R^N}u_0(x)\,\dd x=1$.
\item[(A5)] $\spt(u_0)\subset B_{R_0}$ for some $R_0>0$.
\item[(A6)] $u_0\in D(-\Delta_p)\cap D(-\Delta_p^s)\subset D(A)$,
i.e.\ $-\Delta_p u_0\in L^2(\R^N)$,
$-\Delta_p^s u_0\in L^2(\R^N)$,
and in particular
$u_0\in W^{1,p}(\R^N)\cap W^{s,p}(\R^N)$.
\end{enumerate}
 
\noindent
The parabolic PDE to be solved is:
\begin{equation}
\begin{cases}
\displaystyle\frac{\partial u}{\partial t}(t,x)
=\Delta_p u(t,x)+\Delta_p^s u(t,x),
\qquad(t,x)\in(0,\infty)\times\R^N,\\[8pt]
u(0,x)=u_0(x),\qquad x\in\R^N.
\end{cases}
\label{eq:PDE}
\end{equation}
 
\begin{definition}[Weak solution]
\label{def:weak}
We say $u$ is a \emph{weak solution} of \eqref{eq:PDE} if
\[
u\in L^p_{\mathrm{loc}}\!\bigl(0,\infty;W^{1,p}(\R^N)\bigr)
\cap L^p_{\mathrm{loc}}\!\bigl(0,\infty;W^{s,p}(\R^N)\bigr)
\cap L^\infty_{\mathrm{loc}}\!\bigl(0,\infty;L^2(\R^N)\bigr),
\]
and for every $\phi\in C^\infty_c\bigl((0,\infty)\times\R^N\bigr)$:
\begin{align}
&\int_0^\infty\!\int_{\R^N}
\Bigl[-u\,\partial_t\phi
+|\nabla u|^{p-2}\nabla u\cdot\nabla\phi\Bigr]
\dd x\,\dd t
\notag\\
&+\frac{1}{2}\int_0^\infty\!\int_{\R^N}\!\int_{\R^N}
\frac{|u(t,x)-u(t,y)|^{p-2}(u(t,x)-u(t,y))(\phi(t,x)-\phi(t,y))}
{|x-y|^{N+sp}}
\,\dd y\,\dd x\,\dd t=0.
\label{eq:weak_form}
\end{align}
\end{definition}

\begin{theorem}[Existence and regularity]
\label{thm:existence_PDE}
Under assumptions \textup{(A1)--(A6)},
problem \eqref{eq:PDE} admits a weak solution
$u:[0,+\infty)\times\R^N\to \R$ satisfying:
\begin{align*}
  &1)\,\,u \in L^\infty\!\bigl(0,T;\,L^2(\R^N)\bigr)\cap L^p\!\bigl(0,T;\,W^{1,p}(\R^N)\bigr)\cap L^p\!\bigl(0,T;\,W^{s,p}(\R^N)\bigr),\\
&2)\,\,u\in C\bigl([0,T];L^2(\R^N)\bigr);\\
&3)\,\,|u(t)|_1\leq|u_0|_1
      \mbox{ for a.e. }t\in[0,T];\\
&4)\,\,u(t,x)\geq 0\;\text{a.e.},\\
&5)\,\,u\in L^\infty((0,T)\times\R^N)\mbox{ and }|u(t)|_\infty\leq|u_0|_\infty\mbox{ for a.e. }t\in(0,T)\\
&6)\,\,\mbox{if moreover }p\ge4,\ sp<2\mbox{ and }N\le p,\mbox{ then }\int_{\R^N}u(t,x)\,\dd x = \int_{\R^N}u_0(x)\dd x,\;\text{ for all }t>0.
\label{eq:reg_6}
\end{align*}
\end{theorem}
 
\begin{proof}\quad\\\\
$\bullet$ \textbf{Proof of existence:}The main idea of the proof is based on the theorem  \cite[Theorem~3.1 page 54]{Brezis1973}. However, in order to do so, we first need to prove that the operator $A:=-\Delta_p-\Delta_p^s$ is maximal monotone (or m-accretive) in $L^2(\R^N)$.\\\\
$\square$\textit{ Step 1:} We prove that $A$ is monotone. Define
\begin{equation}
\Phi(u) =
\begin{cases}
\dfrac{1}{p}\displaystyle\int_{\R^N}|\nabla u|^p\,\dd x
+\dfrac{1}{2p}\displaystyle\iint_{\R^N\times\R^N}
\dfrac{|u(x)-u(y)|^p}{|x-y|^{N+sp}}\,\dd x\,\dd y
& \text{if }u\in V,\\[8pt]
+\infty & \text{otherwise},
\end{cases}
\label{eq:Phi}
\end{equation}
with $V:=W^{1,p}(\R^N)\cap W^{s,p}(\R^N)\cap L^2(\R^N)$. For $u,v\in D(A)$, by integration by parts and
the symmetry of the fractional kernel:
\begin{align}
&\langle Au-Av, u-v \rangle_{L^2}
= \int_{\mathbb{R}^N} \bigl( |\nabla u|^{p-2}\nabla u - |\nabla v|^{p-2}\nabla v \bigr) \cdot \nabla(u-v) \, \mathrm{d} x \notag \\
&\quad + \frac{1}{2} \iint_{\mathbb{R}^N \times \mathbb{R}^N} \begin{aligned}[t]
&\Bigl( |u(x)-u(y)|^{p-2}(u(x)-u(y))- |v(x)-v(y)|^{p-2}(v(x)-v(y)) \Bigr)\\
&\times\frac{\left((u-v)(x) - (u-v)(y)\right)}{|x-y|^{N+sp}} \, \mathrm{d} x \, \mathrm{d} y \ge 0
\end{aligned}
\label{eq:mono}
\end{align}
due to the monotonicity of $t\mapsto|t|^{p-2}t$.
Hence $A$ is monotone.\\\\ 
$\square$\textit{ Step 2:} We prove that $A$ is maximal, that is $R(I+A) = L^2(\R^N)$.
\noindent
Define the operator $A_V:V\to V^{*}$ by:
\begin{equation}
\langle A_V u,v\rangle_{V^{*},V}
:= \int_{\R^N}|\nabla u|^{p-2}\nabla u\cdot\nabla v\,\dd x
+\frac{1}{2}\iint_{\R^N\times\R^N}
\frac{|u(x)-u(y)|^{p-2}(u(x)-u(y))(v(x)-v(y))}
{|x-y|^{N+sp}}\,\dd x\,\dd y.
\label{eq:AV}
\end{equation}
 It is straitforward to check that $A_V$ is monotone and hemicontinuous. Indeed, by dominated convergence, for fixed $u,v,w\in V$ the map
$t\mapsto \langle A_V(u+tv),w\rangle_{V^{*},V}$
is continuous and the monotonicity of
$t\mapsto|t|^{p-2}t$ guarantees that
\begin{align}
&\langle A_V u-A_V v, u-v \rangle_{V^{*},V}
= \int_{\mathbb{R}^N} \bigl( |\nabla u|^{p-2}\nabla u - |\nabla v|^{p-2}\nabla v \bigr) \cdot \nabla(u-v) \, \mathrm{d} x \notag \\
&\quad + \frac{1}{2} \iint_{\mathbb{R}^N \times \mathbb{R}^N} \begin{aligned}[t]
&\Bigl( |u(x)-u(y)|^{p-2}(u(x)-u(y))- |v(x)-v(y)|^{p-2}(v(x)-v(y)) \Bigr)\\
&\times\frac{\left((u-v)(x) - (u-v)(y)\right)}{|x-y|^{N+sp}} \, \mathrm{d} x \, \mathrm{d} y \ge 0.
\end{aligned}
\label{eq:mono_V}
\end{align}
Consequently, according to Browder-Minty's theorem  \cite[Theorem~2.4]{Barbu2010}, $A_V$ is maximal monotone in $V\times V^{*}$. Also, $\tilde{A}_V:=I+A_V$ is maximal monotone and hemicontinuous as sum of two maximal monotone and hemicontknuous operators. Therefore, according to Browder-Minty's theorem  \cite[Theorem~2.4]{Barbu2010}, $\tilde{A}_V$ is maximal monotone
in $V\times V^{*}$. Moreover, it is coercive because
\begin{equation}
\frac{\langle\tilde{A}_V u,u\rangle_{V^{*},V}}{\|u\|_V}
= \frac{\|u\|_{L^2}^2+\|\nabla u\|_{L^p}^p
+\frac{1}{2}[u]_{W^{s,p}}^p}{\|u\|_V}
\to+\infty
\quad\text{as }\|u\|_V\to\infty.
\label{eq:coercive}
\end{equation}
It follows from \cite[Corollary 2.2]{Barbu2010} that $R(\tilde{A}_V)=V^{*}$, which implies that $L^2(\R^N)\subset V^{*}\subset R(\tilde{A}_V)$.\\\\
\noindent
As a consequence, for any $f\in L^2(\R^N)$, there exists $u\in V$
such that $u+A_V u=f$ in $V^{*}$.
Since $f\in L^2(\R^N)$ and $u\in L^2(\R^N)$, we have
$A_V u = f-u\in L^2(\R^N)$,
hence $u\in D(A)$ and $Au=A_V u$.
Therefore $u+Au=f$ and $R(I+A)=L^2(\R^N)$.\\\\
$\square$\textit{ Conclusion:} Using \cite[Theorem 3.1]{Brezis1973} applied to
$H=L^2(\R^N)$ provides the existence of a unique function
$u:[0,+\infty)\to L^2(\R^N)$ satisfying
\eqref{eq:PDE}.\\\\
$\bullet$ \textbf{Proof of 1):}
Set $A=-\Delta_p-\Delta_p^s$ and $D(A)=\{u\in V : A_Vu\in L^2(\R^N)\}\supset D(-\Delta_p)\cap D(-\Delta_p^s)$, where
\begin{align*}
&D(-\Delta_p)=\{u\in W^{1,p}(\R^N)\cap L^2(\R^N)
\mid -\Delta_p u\in L^2(\R^N)\},\\
&\mbox{and }\quad D(-\Delta_p^s)=\{u\in W^{s,p}(\R^N)\cap L^2(\R^N)
\mid -\Delta_p^s u\in L^2(\R^N)\}.
\end{align*}
According to \cite[Theorem 3.1]{Brezis1973}, we know that
\begin{align*}
&\textup{(P1)}\quad
u(t)\in D(A)
\quad\text{for all }t>0,\\
&\textup{(P2)}\quad
\frac{du}{dt}\in L^\infty(0,+\infty;H)
\quad\text{and}\quad
\left\|\frac{du}{dt}\right\|_{L^\infty(0,+\infty;H)}
\leq |Au_0|_H,\\
&\textup{(P3)}\quad
\frac{du}{dt}(t)+Au(t)=0
\quad\text{a.e.\ on }(0,+\infty),\\
&\textup{(P4)}\quad u(0)=u_0,\\
&\textup{(P5)}\quad
\left\|u(t)-\hat{u}(t)\right\|_{L^2(\R^N)}\leq\left\|u(0)-\hat{u}(0)\right\|_{L^2(\R^N)}
\quad\forall\,t\geq 0,
\end{align*}
where $\hat{u}$ refers to another solution of \eqref{eq:PDE} satisfying (P1)-(P3). From (P5), taking the particular case $\hat{u}\equiv0$, we get
\begin{equation*}
\|u(t)\|_{L^2(\R^N)}
\leq\|u_0\|_{L^2(\R^N)}
\quad\forall\,t\geq 0,
\end{equation*}
hence $u\in L^\infty(0,T;L^2(\R^N))$.
\noindent
Now we check that $A=\partial\Phi$. Indeed, $\Phi:L^2(\R^N)\to(-\infty,+\infty]$ is convex
and lower semicontinuous, so $\partial\Phi$ is
maximal monotone in $L^2(\R^N)$, see \cite[Theorem A]{Rockafellar1970}. Also, for $u\in D(A)$ and $v\in V$,
by the convexity inequality
$|b|^p-|a|^p\geq p|a|^{p-2}a\cdot(b-a)$ it holds
\begin{align}
\langle Au,u-v\rangle_{L^2}
&= \int_{\R^N}|\nabla u|^{p-2}\nabla u
\cdot\nabla(u-v)\,\dd x
\notag\\
&\quad+\frac{1}{2}\iint_{\R^N\times\R^N}
\frac{|u(x)-u(y)|^{p-2}(u(x)-u(y))
\bigl((u-v)(x)-(u-v)(y)\bigr)}
{|x-y|^{N+sp}}\,\dd x\,\dd y
\notag\\
&\geq \frac{1}{p}\int_{\R^N}|\nabla u|^p\,\dd x
-\frac{1}{p}\int_{\R^N}|\nabla v|^p\,\dd x
+\frac{1}{2p}\iint_{\R^N\times\R^N}
\frac{|u(x)-u(y)|^p-|v(x)-v(y)|^p}
{|x-y|^{N+sp}}\,\dd x\,\dd y
\notag\\
&= \Phi(u)-\Phi(v),
\label{eq:subdiff}
\end{align}
hence $A\subset\partial\Phi$. Consequently, by definition of a maximal monotone operator, $A=\partial\Phi$.\\\\
Using assumption (A6) and \cite[Theorem 3.2]{Brezis1973}, we deduce that the function $t\mapsto\Phi(u(t))$ is
convex, decreasing and Lipschitz on
every $[\delta,+\infty)$, $\delta>0$, with:
\begin{equation*}
\frac{d^+}{dt}\Phi(u(t))
= -\left\|\frac{d^+u}{dt}(t)\right\|_{L^2}^2
\leq 0
\quad\forall\,t>0.
\end{equation*}
It follows that $\Phi(u(t))\leq\Phi(u(\delta))$ for all
$0<\delta\leq t$ and passing to the limit as $\delta\rightarrow0^{+}$, we obtain
\begin{equation*}
\Phi(u(t))\le\liminf_{\delta\to 0^+}\Phi(u(\delta))=\Phi(u_0)= \frac{1}{p}\|\nabla u_0\|_{L^p}^p
+\frac{1}{2p}[u_0]_{W^{s,p}}^p
< \infty.
\end{equation*}
We conclude that
\begin{align*}
\int_0^T\|\nabla u(t)\|_{L^p}^p\,\dd t\leq p\,\Phi(u_0)\cdot T<\infty\quad\mbox{ and }\quad\int_0^T[u(t)]_{W^{s,p}}^p\,\dd t\leq 2p\,\Phi(u_0)\cdot T<\infty,
\end{align*}
hence $u\in L^p(0,T;W^{1,p}(\R^N))$
and $u\in L^p(0,T;W^{s,p}(\R^N))$.\\\\
\noindent
$\bullet$ \textbf{Proof of 2):} From (P2), $\frac{du}{dt}\in L^\infty(0,+\infty;L^2(\R^N))$, hence
\[
\|u(t)-u(\tau)\|_{L^2(\R^N)}\le|Au_0|_H\,|t-\tau|\qquad\forall\,t,\tau\ge0,
\]
that is $t\mapsto u(t)$ is Lipschitz, and in particular continuous, from $[0,T]$ into $L^2(\R^N)$.\\\\
$\bullet$ \textbf{Proof of 3):} Let $\mathcal{X}_\delta$ be the function
\begin{equation*}
\mathcal{X}_\delta(r) =
\begin{cases}
1 & r\geq\delta,\\
r/\delta & |r|<\delta,\\
-1 & r\leq-\delta.
\end{cases}
\label{eq:Xdelta}
\end{equation*}
Since $\mathcal{X}_\delta(u(t))\in L^2$ for a.e.\ $t\in(0,T)$,
according to \cite[Lemma~4.1]{Barbu2010} we have
\begin{equation*}
\Bigl(\frac{d}{dt}u(t),\mathcal{X}_\delta(u(t))\Bigr)_2
= \frac{d}{dt}\int_{\R^N}j_\delta(u(t,x))\dd x,
\quad\text{a.e. }t\in(0,T),
\end{equation*}
where $\displaystyle j_\delta(v):=\int_0^v\mathcal{X}_\delta(s)\dd s$. Therefore, testing \eqref{eq:PDE} with $\varphi=\mathcal{X}_\delta(u)$, we get
\begin{align}
\frac{d}{dt}\int_{\R^N}j_\delta(u)\dd x
&= -\int_{\R^N}|\nabla u|^p
\mathcal{X}_\delta'(u)\dd x
\notag\\
&\quad -\frac{1}{2}\int_{\R^N}\!\int_{\R^N}
\frac{|u(x)-u(y)|^{p-2}(u(x)-u(y))
(\mathcal{X}_\delta(u(x))-\mathcal{X}_\delta(u(y)))}
{|x-y|^{N+sp}}\dd y\,\dd x.
\label{eq:ddt_jdelta}
\end{align}
Notice that $\mathcal{X}_\delta'=
\begin{cases}
1/\delta & -\delta<r<\delta,\\
0 & \text{elsewhere}
\end{cases}$ is positive, which implies that $(u(x)-u(y))(\mathcal{X}_\delta(u(x))-\mathcal{X}_\delta(u(y)))\geq 0$
for all $x,y\in\R^N$. Consequently
\begin{equation}
\frac{d}{dt}\int_{\R^N}j_\delta(u(t,x))\dd x\leq 0,
\quad\text{a.e. }t>0.
\label{eq:ddt_neg}
\end{equation}
Since $j_\delta(v)\to|v|$ as $\delta\to 0$, integrating \eqref{eq:ddt_neg} over $(0,t)$ and passing to the limit, we deduce
\begin{equation}
\int_{\R^N}|u(t,x)|\dd x
\leq \int_{\R^N}|u_0(x)|\dd x,
\quad\text{a.e. }t\in[0,T].
\label{eq:L1bound}
\end{equation}
$\bullet$ \textbf{Proof of 4):} We apply the same argument as in the previous step with
$\mathcal{X}_\delta$ replaced by the approximation of the derivative of $r\mapsto\max\{-r,0\}$, namely
\begin{equation}
\eta_\delta(r) =
\begin{cases}
-1 & r\leq-\delta,\\
r/\delta & -\delta<r<0,\\
0 & r\geq 0.
\end{cases}
\label{eq:etadelta}
\end{equation}
\noindent
Since $\eta_\delta(u(t))\in L^2$ for a.e.\ $t\in(0,T)$,
according to \cite[Lemma~4.1]{Barbu2010} we have
\begin{equation*}
\Bigl(\frac{d}{dt}u(t),\eta_\delta(u(t))\Bigr)_2
= \frac{d}{dt}\int_{\R^N}J_\delta(u(t,x))\dd x,
\quad\text{a.e. }t\in(0,T),
\end{equation*}
where $\displaystyle J_\delta(v):=\int_0^v\eta_\delta(s)\dd s$. Therefore, testing \eqref{eq:PDE} with $\varphi=\eta_\delta(u)$, we deduce
\begin{align}
\frac{d}{dt}\int_{\R^N}J_\delta(u)\dd x
&= -\int_{\R^N}|\nabla u|^p
\eta_\delta'(u)\dd x
\notag\\
&\quad -\frac{1}{2}\int_{\R^N}\!\int_{\R^N}
\frac{|u(x)-u(y)|^{p-2}(u(x)-u(y))
(\eta_\delta(u(x))-\eta_\delta(u(y)))}
{|x-y|^{N+sp}}\dd y\,\dd x.
\label{eq:pos_ddt}
\end{align}
By positivity of $\eta_\delta'(r) =
\begin{cases}
1/\delta & -\delta<r<0,\\
0 & \text{elsewhere}
\end{cases}$, we have $(u(x)-u(y))(\eta_\delta(u(x))-\eta_\delta(u(y)))\geq 0$, hence
\begin{equation}
\frac{d}{dt}\int_{\R^N}J_\delta(u(t,x))\dd x\leq 0,
\quad\text{a.e. }t>0.
\label{eq:pos_mono}
\end{equation}
Since $J_\delta(v)\to\max\{-v,0\}$ as $\delta\to 0$, integrating \eqref{eq:pos_mono} over $(0,t)$ and passing to the limit, we deduce
\begin{equation}
\int_{\R^N}u^-(t,x)\dd x\leq\int_{\R^N}u_0^-(x)\dd x=0,
\quad\text{a.e. }t\in[0,T].
\label{eq:positivity}
\end{equation}
Hence $u\geq 0$ a.e.\ on $(0,T)\times\R^N$.\\\\
$\bullet$ \textbf{Proof of 5):} Let $M:=|u_0|_\infty$.
Testing \eqref{eq:PDE} with
$\varphi=(u-M)^+:=\max\{u-M,0\}$, we get
\begin{align*}
\frac{1}{2}\frac{d}{dt}\int_{\R^N}((u-M)^+)^2\dd x&= -\int_{\{u>M\}}|\nabla u|^p\dd x\\
&-\frac{1}{2}\int_{\R^N}\!\int_{\R^N}
\frac{|u(x)-u(y)|^{p-2}(u(x)-u(y))((u(x)-M)^+-(u(y)-M)^+)}{|x-y|^{N+sp}}\dd y\,\dd x.
\end{align*}
Since $(r\mapsto(r-M)^+)$ is nondecreasing, then for any $a,b\in\R$ and $M\in\R$:
\begin{equation}
(a-b)\bigl((a-M)^+-(b-M)^+\bigr)\geq 0.
\label{eq:sign_NL}
\end{equation}
Consequently,
\begin{equation}
\frac{d}{dt}\int_{\R^N}((u-M)^+)^2\dd x\leq 0,
\quad\text{a.e. }t>0.
\label{eq:vplus_mono}
\end{equation}
Integrating \eqref{eq:vplus_mono} over $(0,t)$, we deduce
\begin{equation}
\int_{\R^N}(u(t,x)-M)^+{}^2\dd x
\leq\int_{\R^N}(u_0(x)-M)^+{}^2\dd x=0,
\label{eq:upper}
\end{equation}
hence $(u(t,x)-M)^+=0$ a.e. in $\R^N$, that is \ $u\leq M$ a.e. in $\R^N$.
Similarly, testing with $(-u-M)^+$ gives $u\geq -M$ a.e. in $\R^N$, and we conclude that
\begin{equation}
|u(t,x)|\leq|u_0|_\infty,
\quad\text{a.e. }(t,x)\in(0,T)\times\R^N.
\label{eq:Linfty}
\end{equation}

\noindent\textbf{$\bullet$ Proof of 6):} Let $\varphi_n(x) := \eta(|x|^2/n)$ with $\eta \in C^2([0,\infty))$ and $\mathbb{1}_{[0,1]} \le \eta \le \mathbb{1}_{[0,2]}$. Testing \eqref{eq:PDE} with $\varphi_n$ and integrating over $(0,t)$ yields:
\begin{align}
\int_{\mathbb{R}^N} \varphi_n u(t) \, \dd x
&= -\int_0^t \!\! \int_{\mathbb{R}^N} |\nabla u|^{p-2} \nabla u \cdot \nabla\varphi_n \, \dd x \, \dd s \notag \\
&\quad - \frac{1}{2} \int_0^t \!\! \iint_{\mathbb{R}^N \times \mathbb{R}^N} \frac{|u(x)-u(y)|^{p-2} \bigl(u(x)-u(y)\bigr) \bigl(\varphi_n(x)-\varphi_n(y)\bigr)}{|x-y|^{N+sp}} \, \dd y \, \dd x \, \dd s \notag \\
&\quad + \int_{\mathbb{R}^N} \varphi_n u_0 \, \dd x.
\label{eq:mass_approx}
\end{align}

\noindent We pass to the limit $n \to \infty$ in both integrals. Using H\"older's inequality alongside the sequence properties:
\begin{itemize}
    \setlength{\itemsep}{1.5ex}
    \item[$\bullet$] $|\nabla\varphi_n(x)| \le \dfrac{C}{\sqrt{n}} \mathbb{1}_{B_n}(x)$,
    \item[$\bullet$] $B_n := \bigl\{ x \in \mathbb{R}^N : \sqrt{n} \le |x| \le \sqrt{2n} \bigr\}$,
\end{itemize}
\medskip
we obtain:
\begin{align*}
\Biggl| \int_0^t \!\! \int_{\mathbb{R}^N} |\nabla u|^{p-2} \nabla u \cdot \nabla\varphi_n \, \dd x \, \dd s \Biggr|
&\le \frac{C}{\sqrt{n}} \int_0^t \Biggl( \int_{|x|\ge\sqrt{n}} |\nabla u|^p \, \dd x \Biggr)^{\frac{p-1}{p}} \bigl( \operatorname{Vol}(B_n) \bigr)^{\frac{1}{p}} \dd s \\
&\le C \, n^{\frac{N-p}{2p}} \int_0^t \Biggl( \int_{|x|\ge\sqrt{n}} |\nabla u|^p \, \dd x \Biggr)^{\frac{p-1}{p}} \dd s \quad \xrightarrow[n \to \infty]{} \quad 0,
\end{align*}
\noindent By property 1) and $N \le p$, the dominated convergence theorem ensures this last integral vanishes.

\medskip

\noindent For the nonlocal term, the pointwise bounds:
\[
\begin{cases}
    |\varphi_n(x)-\varphi_n(y)| \le \min\{1, C|x-y|\}, \\[2ex]
    |u(x)-u(y)|^{\frac{p}{2}-1} \le M^{\frac{p}{2}-2}\bigl(u(x)+u(y)\bigr) \quad (\text{since } 0 \le u \le M, \, p \ge 4),
\end{cases}
\]
directly yield the following uniform bound $G$ for the integrand:
\[
G := M^{\frac{p}{2}-2} \, \dfrac{|u(x)-u(y)|^{\frac{p}{2}} \bigl(u(x)+u(y)\bigr) \min\{1,C|x-y|\}}{|x-y|^{N+sp}}.
\]
By Cauchy--Schwarz's inequality,
\begin{align*}
\int_0^t \!\! \iint_{\mathbb{R}^N\times\mathbb{R}^N} G\, \dd y \, \dd x \, \dd s
&\le M^{\frac{p}{2}-2} \Biggl( \int_0^t \!\! \iint_{\mathbb{R}^N\times\mathbb{R}^N} \frac{|u(x)-u(y)|^p}{|x-y|^{N+sp}} \, \dd y \, \dd x \, \dd s \Biggr)^{\frac{1}{2}} \\
&\quad \times \Biggl( \int_0^t \!\! \iint_{\mathbb{R}^N\times\mathbb{R}^N} \frac{2\bigl(u(x)^2+u(y)^2\bigr) \min\{1,C^2|x-y|^2\}}{|x-y|^{N+sp}} \, \dd y \, \dd x \, \dd s \Biggr)^{\frac{1}{2}} \\[2ex]
&\le M^{\frac{p}{2}-2} \, \|u\|_{L^p(0,t;W^{s,p})}^{\frac{p}{2}} \Biggl( 4t \sup_{s\in(0,t)} \|u(s)\|_{L^2}^2 \, \omega_N \Bigl( \frac{C^{sp}}{2-sp} + \frac{C^{sp}}{sp} \Bigr) \Biggr)^{\frac{1}{2}} < \infty,
\end{align*}
since $sp<2$. Since $\varphi_n(x)-\varphi_n(y) \to 0$ pointwise, the nonlocal term also vanishes by the dominated convergence theorem.

\smallskip

\noindent The integrability $u(t) \in L^1(\mathbb{R}^N)$ for every $t \in [0,T]$ is established through the following sequence:
\[
\begin{cases}
    t_k \to t \quad \text{satisfying property 3)}, \\[1.5ex]
    u(t_k) \to u(t) \quad \text{in } L^2(\mathbb{R}^N) \text{ and a.e. along a subsequence (by 2)}, \\[2ex]
    \displaystyle\int_{\mathbb{R}^N} |u(t,x)| \, \dd x \le \liminf_{k \to \infty} \int_{\mathbb{R}^N} |u(t_k,x)| \, \dd x \le \int_{\mathbb{R}^N} |u_0(x)| \, \dd x \quad \text{(by Fatou's lemma)}.
\end{cases}
\]

\medskip

\noindent Letting $n \to \infty$ in \eqref{eq:mass_approx}, the dominated convergence theorem immediately yields:
\begin{align*}
\int_{\mathbb{R}^N} u(t,x) \, \dd x
&= \lim_{n \to \infty} \int_{\mathbb{R}^N} \varphi_n(x) u(t,x) \, \dd x \\[2ex]
&= \lim_{n \to \infty} \int_{\mathbb{R}^N} \varphi_n(x) u_0(x) \, \dd x \\[2ex]
&= \int_{\mathbb{R}^N} u_0(x) \, \dd x.
\end{align*}
This completes the proof.

\end{proof}
\subsection{Existence of a distributional solution for the PDE \eqref{eq:PDE_distributional}}
Using the chain rule $\Delta_pu=\Delta(|\nabla u|^{p-2}u)
-\operatorname{div}(\nabla(|\nabla u|^{p-2})u)$, we may rewrite \eqref{eq:PDE} as follows
\begin{equation}
\begin{cases}
\displaystyle
\frac{\partial u}{\partial t}(t,x)
= \Delta(|\nabla u|^{p-2}u)
-\operatorname{div}(\nabla(|\nabla u|^{p-2})u) + \Delta_p^su,
\quad (t,x)\in(0,\infty)\times\R^N,
\\[8pt]
u(0,x) = u_0(x), \qquad x\in\R^N.
\end{cases}
\label{eq:PDE_distributional}
\end{equation}
\begin{definition}\label{def:23}
A function $u$ is called a \emph{distributional solution} of \eqref{eq:PDE_distributional} if it
satisfies
\begin{enumerate}
\item[(C1)] $u\in L^1_{\mathrm{loc}}((0,\infty);W^{1,1}_{\mathrm{loc}}(\R^N))\cap L^p_{\mathrm{loc}}((0,\infty);W^{s,p}_{\mathrm{loc}}(\R^N))$,
    \item[(C2)] $\nabla u\in L^{p-1}_{\mathrm{loc}}((0,\infty);
L^{p-1}_{\mathrm{loc}}(\R^N;\R^N))$
\item[(C3)] $|\nabla u|^{p-2}\in L^1_{\mathrm{loc}}((0,\infty);
W^{1,1}_{\mathrm{loc}}(\R^N))$,
\item[(C4)] $|\nabla u|^{p-2}u\in L^1_{\mathrm{loc}}((0,\infty)\times\R^N)$,
\item[(C5)] $\nabla(|\nabla u|^{p-2})u\in L^1_{\mathrm{loc}}((0,\infty)
\times\R^N;\R^N)$,
\item[(C6)] $\displaystyle(1\wedge|z|)\frac{|u(t,x)-u(t,x+z)|^{p-2}}{|z|^{N+sp}}u(t,x)
\in L^1_{\mathrm{loc}}\bigl((0,\infty);L^1(\R^N
\times(\R^N\setminus\{0\}))\bigr)$,
\end{enumerate}
and for all
$\varphi\in C^2_0((0,\infty)\times\R^N)$:
\begin{align}
&\int_0^\infty\!\int_{\R^N}
u\Bigl(\partial_t\varphi+|\nabla u|^{p-2}\Delta\varphi
+\nabla(|\nabla u|^{p-2})\cdot\nabla\varphi\Bigr)\dd x\,\dd t
\notag\\
&+\int_0^\infty\!\int_{\R^N}\!\int_{\R^N\setminus\{0\}}
u(t,x)\bigl[\varphi(t,x+z)-\varphi(t,x)\bigr]
\nu_{t,x}(\dd z)\,\dd x\,\dd t = 0.
\label{eq:23}
\end{align}
\end{definition}
\noindent Now consider the following assumptions:
\begin{enumerate}
    \item[(B1)]\,\,$p\geq 4$, $s\in(1/p,1)$, $sp>1$;
    \item[(B2)]\,\,$sp < 2$;
\item[(B3)]\,\,$sp\in(1,N)$;
\item[(B4)]\,\,$N\le p$.
\end{enumerate}
Assumption (B2) enables the integrability of the L\'evy kernel near the origin, (B3) the Sobolev embedding $W^{s,p}(\R^N)\hookrightarrow
L^q(\R^N)$ for $1\leq q\leq p^*_s:=Np/(N-sp)$, and (B4) will be used to prove the mass conservation, see Theorem~\ref{thm:existence_PDE}. The condition $p\ge4$ in (B1) is used in the mass conservation and in the proof of (H2), through the inequality $|u(x)-u(y)|^{\frac p2-1}\le M^{\frac p2-2}(u(x)+u(y))$ for $0\le u\le M$ and H\"older's inequality with exponents $\frac{p}{p-4}$ and $\frac p4$, as well as in the convexity arguments of the Appendix.
\begin{lemma}[Integrability of the compensator]
\label{lem:compensator}
Under assumptions (A)+(B), let $u$ be the solution of \eqref{eq:PDE} given by Theorem~\ref{thm:existence_PDE} and $M:=|u_0|_\infty$. Then for every $0\le a<b\le T$,
\begin{equation}
\int_a^b\!\int_{\R^N}\left(\int_{|z|\le1}|z|\,\nu_{t,x}(\dd z)\right)u(t,x)\,\dd x\,\dd t
\le C\,M^{\frac{p-2}{p}}\,\|\nabla u\|_{L^p((0,T)\times\R^N)}^{p-2}\,(b-a)^{\frac2p},
\label{eq:compensator}
\end{equation}
with $C=C(N,p,s)$. Consequently,
\[
\int_{\R^N\setminus\{0\}}(1\wedge|z|)\,\nu_{t,x}(\dd z)<\infty\quad\mbox{ for }\dd t\otimes u(t,x)\dd x\mbox{-a.e. }(t,x),
\]
and for every $\ell\in(0,1]$ the vector
$\displaystyle c_\ell(t,x):=\int_{|z|\le\ell}z\,\nu_{t,x}(\dd z)$ is well defined and satisfies
\[
\int_a^b\!\int_{\R^N}|c_\ell(t,x)|\,u(t,x)\,\dd x\,\dd t\le C\,M^{\frac{p-2}{p}}\,\|\nabla u\|_{L^p((0,T)\times\R^N)}^{p-2}\,(b-a)^{\frac2p}.
\]
\end{lemma}
\begin{proof}  \mbox{} \\[2ex]
\noindent\textbf{Step 1: A pointwise bound on the kernel.}

\medskip

\noindent Since $u(t)\in W^{1,p}(\mathbb{R}^N)$ for a.e.\ $t$, by density of $C^\infty_c(\mathbb{R}^N)$ in $W^{1,p}(\mathbb{R}^N)$, for every $z\in\mathbb{R}^N$ we have, for a.e.\ $x\in\mathbb{R}^N$,
\[
u(t,x+z)-u(t,x)=\int_0^1\nabla u(t,x+\theta z)\cdot z\,\dd\theta.
\]
Therefore, since $p-2\ge1$, Jensen's inequality gives
\[
\frac{|u(t,x)-u(t,x+z)|^{p-2}}{|z|^{N+sp}}
\le |z|^{p-2-N-sp}\int_0^1|\nabla u(t,x+\theta z)|^{p-2}\,\dd\theta .
\]

\medskip

\noindent\textbf{Step 2: Integration and Hölder bounds.}

\medskip

\noindent Consequently, using Fubini's theorem, H\"older's inequality with exponents $\dfrac{p}{p-2}$ and $\dfrac{p}{2}$, and the translation invariance of the Lebesgue measure, we obtain:
\begin{align*}
\int_{\mathbb{R}^N} \!\! \int_{|z|\le 1} &|z| \, \nu_{t,x}(\dd z) \, u(t,x) \, \dd x \\
&\le \int_{|z|\le 1} |z|^{p-1-N-sp} \int_0^1\Biggl( \int_{\mathbb{R}^N} |\nabla u(t,x+\theta z)|^{p-2}\, u(t,x) \, \dd x \Biggr)\dd\theta\, \dd z \\[2ex]
&\le \|\nabla u(t)\|_{L^p(\mathbb{R}^N)}^{p-2} \, \|u(t)\|_{L^{p/2}(\mathbb{R}^N)} \int_{|z|\le 1} |z|^{p-1-N-sp} \, \dd z \\[2ex]
&\le \|\nabla u(t)\|_{L^p(\mathbb{R}^N)}^{p-2} \, M^{\frac{p-2}{p}} \, \frac{\omega_N}{p-1-sp}.
\end{align*}

\smallskip

\noindent This last inequality is justified by the condition $p-1-sp>0$ and by the uniform bounds from Theorem~\ref{thm:existence_PDE} (items 3--5):
\[
\|u(t)\|_{L^{p/2}(\mathbb{R}^N)}^{p/2} \le M^{\frac{p}{2}-1} \int_{\mathbb{R}^N} u(t,x) \, \dd x \le M^{\frac{p}{2}-1}.
\]

\smallskip

\noindent Integrating over time $(a,b)$ and using H\"older's inequality in time, we obtain:
\[
\int_a^b \|\nabla u(t)\|_{L^p(\mathbb{R}^N)}^{p-2} \, \dd t \le (b-a)^{\frac{2}{p}} \, \|\nabla u\|_{L^p((0,T)\times\mathbb{R}^N)}^{p-2},
\]
which gives \eqref{eq:compensator} with $C=\frac{\omega_N}{p-1-sp}$.

\medskip

\noindent\textbf{Step 3: Remaining claims.}

\medskip

\noindent Since $0 \le u \le M$, the tail of the integral can be bounded as follows:
\[
\int_{|z|>1} \nu_{t,x}(\dd z) \le M^{p-2} \omega_N \int_1^\infty r^{-1-sp} \, \dd r = \frac{M^{p-2} \omega_N}{sp}.
\]

\smallskip

\noindent Together with \eqref{eq:compensator}, this gives the remaining claims, since the compensator satisfies:
\[
|c_\ell(t,x)| \le \int_{|z|\le 1} |z| \, \nu_{t,x}(\dd z).
\]
\end{proof}
\begin{remark} Lemma~\ref{thm:levy} only guarantees the finiteness of the second moment:
\[
\displaystyle\int_{\{|z|\le1\}} |z|^2 \, \nu_{t,x}(\dd z) < \infty.
\]

\medskip

\noindent The finiteness of the first moment is not a mere consequence of the boundedness of $u$. Indeed, using only the crude bound $|u(t,x)-u(t,x+z)| \le M$, we get a useless estimate since $sp > 1$:
\[
\displaystyle\int_{\{|z|\le1\}} |z| \, \nu_{t,x}(\dd z) \le M^{p-2} \omega_N \displaystyle\int_0^1 r^{-sp} \, \dd r = +\infty.
\]

\medskip

\noindent To overcome this, Lemma~\ref{lem:compensator} exploits the Sobolev regularity $u(t) \in W^{1,p}(\mathbb{R}^N)$. With this regularity, the density of $\nu_{t,x}$ behaves like $|z|^{p-2-N-sp}$, so that $|z|\,\nu_{t,x}(\dd z)$ behaves like $|z|^{p-1-N-sp}\dd z$, which yields a convergent integral since $p-1-sp > 0$:
\[
\displaystyle\int_0^1 r^{p-2-sp} \, \dd r < \infty.
\]

\medskip

\noindent This finiteness is crucial, and Lemma~\ref{lem:compensator} is used in four distinct places:

\begin{enumerate}
    \setlength{\itemsep}{2.5ex}

    \item[\textbf{a)}] \textit{The generator.} The superposition principle of \cite{RXZ2020} is formulated using the compensated jump operator:
    \[
    \int \Theta_f(x;z) \, \nu_{t,x}(\dd z), 
    \qquad \text{where} \qquad 
    \Theta_f(x;z) := f(x+z) - f(x) - \mathbf{1}_{\{|z|\le\ell\}} z \cdot \nabla f(x).
    \]
    Our operator, however, involves the uncompensated integral:
    \[
    \int \bigl[ f(x+z) - f(x) \bigr] \, \nu_{t,x}(\dd z).
    \]
    Both coincide up to the drift term $c_\ell(t,x) \cdot \nabla f(x)$, which is finite by Lemma~\ref{lem:compensator}. This explains why the overall drift is $b = \overline{b} + c_\ell$ in \eqref{particular_case}.

    \item[\textbf{b)}] \textit{Condition (C6) and the symmetrization identity \eqref{eq:sym}.} In this context, Fubini's theorem requires the integrability of the measure:
    \[
    (1 \wedge |z|) \, \nu_{t,x}(\dd z) \, u(t,x) \, \dd x \, \dd t.
    \]

    \item[\textbf{c)}] \textit{Assumption (H2).} The drift $b = \overline{b} + c_\ell$ must be integrated.

    \item[\textbf{d)}] \textit{The SDE \eqref{eq:SDE}.} The jumps are not compensated. The sum of the small jumps:
    \[
    \int z \, \mathcal{N}(\dd t,\dd z)
    \]
    is finite precisely because their first moment is finite.
\end{enumerate}
\end{remark}
\begin{remark} Lemma~\ref{lem:compensator} is completely independent of the weighted second-order regularity detailed in the Appendix. It only relies on the standard bounds:
\[
u \in L^p\bigl(0,T; W^{1,p}(\mathbb{R}^N)\bigr) \cap L^\infty 
\qquad \text{and} \qquad 
\int_{\mathbb{R}^N} u(t,x) \, \dd x \le 1.
\]
\end{remark}

\begin{theorem}[Distributional formulation]
\label{thm:distributional}
Under assumptions (A)+(B), the solution $u$ of \eqref{eq:PDE} given by Theorem~\ref{thm:existence_PDE} is a distributional solution of \eqref{eq:PDE_distributional} in the sense of Definition~\ref{def:23}.
\end{theorem}
\begin{proof}
According to Definition \ref{def:23}, it suffices to prove that the solution $u$ given by Theorem~\ref{thm:existence_PDE} satisfies properties (C1)-(C6) and \eqref{eq:23}.\par\bigskip\noindent
$\bullet$ \noindent\textbf{Proof of (C1) - (C2).} It is a direct consequence of Theorem \ref{thm:existence_PDE}.\\\\
$\bullet$ \noindent\textbf{Proof of (C3).} Let $[a,b]\subset(0,\infty)$ and $K\subset\R^N$, $K$ compact.
According to Proposition~\ref{prop:hybrid31} in the Appendix, we know that
\begin{equation}
C_0:=\int_a^b\!\int_K|\nabla u|^{p-2}|D^2u|^2\,\dd x\,\dd t<\infty.
\label{eq:C0}
\end{equation}
Therefore, using the chain rule and Hölder's inequalities, we get
\begin{align}\label{chain_rule}
\begin{aligned}
\int_a^b\||\nabla u|^{p-2}\|_{W^{1,1}(K)}\dd t&:=\int_a^b\!\int_K|\nabla u|^{p-2}\,\dd x\,\dd t+\int_a^b\!\int_K|\nabla(|\nabla u|^{p-2})|\,\dd x\,\dd t\\
&\le\int_a^b\!\int_K|\nabla u|^{p-2}\,\dd x\,\dd t+(p-2)\int_a^b\!\int_K\bigl(|\nabla u|^{\frac{p-2}{2}}|D^2u|\bigr)
|\nabla u|^{\frac{p-4}{2}}\,\dd x\,\dd t\\
&\le|K|^{2/p}(b-a)^{2/p}\|u\|_{L^p((0,\infty);W^{1,p}(\R^N))}^{p-2}+(p-2)C_0^{1/2}\Bigl(\int_a^b\!\int_K|\nabla u|^{p-4}
\,\dd x\,\dd t\Bigr)^{1/2}\\
&\le|K|^{2/p}(b-a)^{2/p}\left(\|u\|_{L^p((0,\infty);W^{1,p}(\R^N))}^{p-2}+(p-2)C_0^{1/2}\|u\|_{L^p((0,\infty);W^{1,p}(\R^N))}^{\frac{p-4}{2}}\right)\\
&<\infty,
\end{aligned}
\end{align}
hence $|\nabla u|^{p-2}\in L^1_{\mathrm{loc}}((0,\infty);
W^{1,1}_{\mathrm{loc}}(\R^N))$.
\\\\
$\bullet$ \noindent\textbf{Proof of (C4) - (C5).}
Using $\|u(t,\cdot)\|_{L^\infty}\leq M$, we get
\begin{equation}
\begin{aligned}
&\int_a^b\!\int_K |\nabla u|^{p-2}|u|\,\dd x\,\dd t+\int_a^b\!\int_K |\nabla(|\nabla u|^{p-2})||u|\,\dd x\,\dd t\\
&\leq M\int_a^b\!\int_K |\nabla u|^{p-2}\,\dd x\,\dd t+M\int_a^b\!\int_K |\nabla(|\nabla u|^{p-2})|\,\dd x\,\dd t<\infty,
\end{aligned}
\label{eq:R2R3}
\end{equation}
hence the claim.\\\\
$\bullet$ \noindent\textbf{Proof of (C6).} It is a direct consequence of Lemma~\ref{lem:compensator}.\\\\
\noindent\textbf{The identity \eqref{eq:23}.} 

\medskip

\noindent Let $\varphi \in C^2_0\bigl((0,\infty)\times\mathbb{R}^N\bigr)$.
\vspace{0.3cm}
\begin{itemize}
    \setlength{\itemsep}{2ex}
    
    \item[$\bullet$] \textit{Local terms:} The identity \eqref{eq:23} follows from \eqref{eq:weak_form} by applying the chain rule:
    \[
    \Delta_p u = \Delta\bigl(|\nabla u|^{p-2}u\bigr) - \operatorname{div}\bigl(\nabla(|\nabla u|^{p-2})u\bigr),
    \]
    which is rigorously justified by conditions (C3)--(C5).
    
    \item[$\bullet$] \textit{Nonlocal term:} We introduce the symmetric kernel:
    \[
    K(x,y) := \frac{|u(t,x)-u(t,y)|^{p-2}}{|x-y|^{N+sp}}.
    \]
\end{itemize}

\smallskip

\noindent We use the following bound for the test function differences:
\[
|\varphi(t,y)-\varphi(t,x)| \le \|\nabla\varphi\|_\infty |x-y| \mathbf{1}_{\{|x-y|\le 1\}} + 2\|\varphi\|_\infty \mathbf{1}_{\{|x-y|>1\}}.
\]

\noindent Using this bound, condition (C6) shows that:
\[
\int_0^\infty \!\! \iint_{\mathbb{R}^N\times\mathbb{R}^N} |\varphi(t,y)-\varphi(t,x)| \, K(x,y) \, u(t,x) \, \dd y \, \dd x \, \dd t < \infty.
\]

\smallskip

\noindent Therefore, Fubini's theorem applies. By exchanging $x$ and $y$, we deduce:
\begin{align}
\iint_{\mathbb{R}^N\times\mathbb{R}^N} &\bigl(\varphi(y)-\varphi(x)\bigr) K(x,y) \, u(x) \, \dd y \, \dd x \nonumber \\
&= \dfrac{1}{2} \iint_{\mathbb{R}^N\times\mathbb{R}^N} \bigl(\varphi(y)-\varphi(x)\bigr) K(x,y) \bigl(u(x)-u(y)\bigr) \, \dd y \, \dd x \nonumber \\[2ex]
&= -\dfrac{1}{2} \iint_{\mathbb{R}^N\times\mathbb{R}^N} \frac{|u(x)-u(y)|^{p-2}\bigl(u(x)-u(y)\bigr)}{|x-y|^{N+sp}} \bigl(\varphi(x)-\varphi(y)\bigr) \, \dd y \, \dd x,
\label{eq:sym}
\end{align}
which is exactly the opposite of the nonlocal term of \eqref{eq:weak_form}. 

\medskip

\noindent Hence, multiplying \eqref{eq:weak_form} by $-1$, \eqref{eq:23} holds. This completes the proof.
\end{proof}

\subsection{Existence of a weak solution for the Fokker-Planck-Kolmogorov Equation (FPKE)}
Following R\"ockner--Xie--Zhang \cite{RXZ2020},
we consider the general non-local Fokker--Planck--Kolmogorov
equation (FPKE):
\begin{equation}
\left\{
\begin{aligned}
&\partial_t\mu_t = \mathcal{L}_t^*\mu_t, \\[6pt]
&\text{with }  \mathcal{L}_t := \mathcal{A}_t + \mathcal{B}_t + \mathcal{N}_t
\end{aligned}
\right.
\label{eq:FPE_general}
\end{equation}
where, for $f\in C^2_b(\R^N)$:
\begin{equation}
\left\{
\begin{aligned}
\mathcal{A}_t f(x)
&:= \text{tr}(a(t,x) \cdot \nabla^2 f(x)) = \sum_{i,j=1}^N a_{ij}(t,x)\,\partial^2_{ij}f(x), \\[4pt]
\mathcal{B}_t f(x)
&:= b(t,x) \cdot \nabla f(x) = \sum_{i=1}^N b_i(t,x)\,\partial_i f(x), \\[4pt]
\mathcal{N}_t f(x)
&:= \int_{\R^N\setminus\{0\}} \Theta_f(x;z)\,\nu_{t,x}(\dd z), \quad \Theta_f(x;z):=f(x+z)-f(x)-\mathbf{1}_{|z|\leq\ell}z\cdot\nabla f(x)
\end{aligned}
\right.
\label{eq:operators}
\end{equation}
with $\{\nu_{t,x}\}_{t\ge0,x\in\R^N}$ a family of Lévy measures over $\R^N$, $a(t,x) = (a_{ij}(t,x))_{1\leq i,j\leq N}$
a symmetric positive semidefinite matrix-valued function in $\mathbb{M}_{N\times N}(\R)$ and $b(t,x)=(b_i(t,x))_{1\leq i\leq N}$ a vector field in $\R^N$.\\\\
In our particular case, we set
\begin{equation}\label{particular_case}
a(t,x):=\overline{a}(t,x)I,\quad b(t,x):=\overline{b}(t,x)+c_\ell(t,x)\quad\mbox{ and }\quad\nu_{t,x}(dz)=\frac{|u(t,x)-u(t,x+z)|^{p-2}}{|z|^{N+sp}}dz,
\end{equation}
where $u$ is the solution of \eqref{eq:PDE} given by Theorem~\ref{thm:existence_PDE} and
\begin{equation*}
\overline{a}(t,x)= |\nabla u(t,x)|^{p-2},\quad\overline{b}(t,x)= \nabla(|\nabla u(t,x)|^{p-2})\quad\mbox{ and }\quad c_\ell(t,x):=\int_{|z|\le\ell}z\,\nu_{t,x}(\dd z).
\end{equation*}
By Lemma~\ref{lem:compensator}, $c_\ell(t,x)$ is well defined for $\dd t\otimes u(t,x)\dd x$-a.e. $(t,x)$.
Assembling the three parts, $\mathcal{L}_t$ may be defined as follows
\begin{equation}
\mathscr{L}_t f(x)
= |\nabla u|^{p-2}\Delta f
+\nabla(|\nabla u|^{p-2})\cdot\nabla f
+\int_{\R^N}
\frac{|u(x)-u(x+z)|^{p-2}}{|z|^{N+sp}}
[f(x+z)-f(x)]\dd z.
\label{eq:L_assembly}
\end{equation}
Notice that, by Lemma~\ref{lem:compensator}, the last integral is absolutely convergent for $\dd t\otimes u(t,x)\dd x$-a.e. $(t,x)$, and that
\[
\int_{\R^N}\frac{|u(x)-u(x+z)|^{p-2}}{|z|^{N+sp}}[f(x+z)-f(x)]\dd z=\int_{\R^N\setminus\{0\}}\Theta_f(x;z)\,\nu_{t,x}(\dd z)+c_\ell(t,x)\cdot\nabla f(x),
\]
so that $\mathscr{L}_t=\mathcal{A}_t+\mathcal{B}_t+\mathcal{N}_t$ with $a$, $b$ and $\nu_{t,x}$ given by \eqref{particular_case}.
\begin{definition}[Weak solution of \eqref{eq:FPE_general}]
\label{def:weak_FPE}
We call $\mu_t$ a weak solution of the non-local FPE  \eqref{eq:FPE_general}
, if there exists a fixed number $\ell \in \left]0, 1/\sqrt{2}\right]$ such that for any $R>0$
and $t>0$:
\begin{equation*}
\begin{aligned}
&1)\,\,\int_0^t\!\int_{\R^N}\mathbf{1}_{B_R}(x)
\Bigl(|\overline{a}(s,x)|+|b(s,x)|+|g_{\nu}(s,x)|\Bigr)\mu_s(\dd x)\,\dd s<\infty,\\[4pt]
&2)\,\,\int_0^t\!\int_{\R^N}
\Bigl(\nu_{s,x}(B^c_{\ell\vee(|x|-R)})
+\mathbf{1}_{B_R}(x)\nu_{s,x}(B_\ell^c)\Bigr)
\mu_s(\dd x)\,\dd s<\infty,
\end{aligned}
\end{equation*}
where $\displaystyle g_{\nu}(t,x):=\int_{B_\ell}|z|^2\nu_{t,x}(dz)$ for each $x\in\R^N$, $\displaystyle \nu_{t,x}\left(E\right):=\int_{E}\nu_{t,x}(dz)$ and $\displaystyle B_\ell := \{z \in \mathbb{R}^N : |z| < \ell\}$. Also, for all $f\in C^2_c(\R^N)$ :
\begin{equation}
\mu_t(f) = \mu_0(f)
+\int_0^t\mu_s(\mathcal{L}_s f)\,\dd s,
\label{eq:weak_FPE}
\end{equation}
where $\displaystyle\mu_t(f):=\int_{\R^N}f(x)\mu_t(\dd x)$
and $\mathcal{L}_s f$ is given by \eqref{eq:L_assembly}.  
\end{definition}
\begin{remark}\quad\\
\begin{enumerate}
\item[a)] Since $\{\nu_{t,x}\}_{t\ge0,x\in\R^N}$ is a family of Lévy measures over $\R^N$ in the sense of Definition \ref{eq:levy_def}, we check that
\begin{equation*}
g_{\nu}(t,x):=\int_{B_\ell}|z|^2\nu_{t,x}(dz)<\infty\quad\mbox{ and }\quad\nu_{s,x}(B^c_{\ell}) < \infty,
\end{equation*}
which guarantees that the integrands in assumptions 1) and 2) are well defined.
\item[b)] Assumptions 1) - 2) are necessary to ensure that $\displaystyle\int_0^t\mu_s(\mathcal{L}_s f)\,\dd s$ is well defined. Indeed, let $f\in C^2_c(\R^N)$ such that $\mbox{supp }f\subset B_R$. In $\R^N\times B_\ell$, if $|x|\ge R+\ell$, then $\nabla f(x)=\overrightarrow{0}$ and $|x+z|\ge|x|-|z|\ge R+\ell-|z|>R$. Therefore, using Taylor's formula, it holds for some $\xi\in(x,x+z)$
\begin{align}\label{un-2}
\begin{aligned}
\left|\Theta_f(x;z)\right|\le\left\|D^2f\right\|_{\infty}|z|^2\mathbb{1}_{\left\{|x|<R+\ell\right\}}(x).
\end{aligned}
\end{align}
In $\R^N\times B_\ell^c$, $\Theta_f(x;z)=f(x+z)\mathbf{1}_{\{|x+z|<R\}}-f(x)\mathbf{1}_{\{|x|<R\}}.$
If $x+z\in B_R$, then $|x|\le|x+z|+|z|\le R+|z|$, that is $|z|\ge|x|-R$. Therefore 
\begin{align}\label{deux-2}
\begin{aligned}
\left|\Theta_f(x;z)\right|\le\|f\|_\infty
\left(
\mathbf{1}_{\{|x|-R<|z|\}}
+
\mathbf{1}_{\{|x|<R\}}
\right).
\end{aligned}
\end{align}
Combining \eqref{un-2}-\eqref{deux-2}, we deduce
\[
\left|\Theta_f(x;z)\right|
\le
\|f\|_\infty
\mathbf{1}_{\{|z|>\ell\}}
\left(
\mathbf{1}_{\{|x|-R<|z|\}}
+
\mathbf{1}_{\{|x|<R\}}
\right)
+
\|\nabla^2 f\|_\infty
\mathbf{1}_{\{|z|\le \ell\}}
|z|^2
\mathbf{1}_{\{|x|<R+\ell\}},
\]
hence
\begin{align*}
&\left|
\int_0^t\mu_s(\mathcal{L}_s f)\,\dd s
\right|\\
&\le\left|
\int_0^t\!\int_{B_R}
\overline{a}(s,x)\,\Delta f(s,x)\,\mu_s(\dd x)\,\dd s\right|+\left|\int_0^t\!\int_{B_R}
b(s,x)\cdot\nabla f(s,x)\,\mu_s(\dd x)\,\dd s\right|\\
&+\left|\int_0^t\!\int_{\R^N}\int_{\R^N\backslash\{0\}}
\Theta_f(x;z)\,\nu_{s,x}(dz)\mu_s(\dd x)\,\dd s
\right|\\
&\le\|f\|_{C_c^2}
\int_0^t\!\int_{B_R}
\left(\left|\overline{a}(s,x)\right|+\left|b(s,x)\right|\right)\,\mu_s(\dd x)\,\dd s\\
&+C\left(\int_0^t
\int_{\mathbb{R}^N}
\left[
\nu_{s,x}\!\left(
B_{\ell\vee(|x|-R)}^c
\right)
+
\mathbf{1}_{B_R}(x)\,
\nu_{s,x}\!\left(B_\ell^c\right)
\right]
\mu_s(dx)\,ds+
\int_0^t
\int_{\mathbb{R}^N}
\mathbf{1}_{B_{R+\ell}}(x)\,
g_{\nu}(s,x)\,
\mu_s(dx)\,ds\right)\\
&<\infty.
\end{align*}
\item[c)] The assumption $\ell \le 1/\sqrt{2}$ is a technical condition imposed (without loss of generality) to ensure the validity of condition (2.2) in \cite[Lemma 2.3]{RXZ2020}.
\end{enumerate}
\end{remark}
\begin{theorem}[From the PDE to the FPKE]\label{thm:bridge_complete}
Under assumptions (A)+(B), if $u$ is the solution of \eqref{eq:PDE} given by Theorem~\ref{thm:existence_PDE}, then $\mu_t:=u(t,\cdot)\dd x$
is a solution of \eqref{eq:FPE_general} in the sense of Definition~\ref{def:weak_FPE}.
\end{theorem}
\begin{proof} Set $\mu_t:=u(t,.)dx$. According to Definition \ref{def:weak_FPE}, it suffices to prove that $u$ satisfies properties 1) - 2) and equality \eqref{eq:weak_FPE}.\\\\
$\bullet$ \textbf{Proof of 1):} Using H\"older's inequality, Theorem \ref{thm:existence_PDE} and Lemma~\ref{thm:levy}, we get:
\begin{align}
\begin{aligned}
&\int_0^t\!\int_{\R^N}\mathbf{1}_{B_R}(x)
\Bigl(|a_{ij}(s,x)|+|b_i(s,x)|+|g_{\nu}(s,x)|\Bigr)\mu_s(\dd x)\,\dd s\\
&=\int_0^t\!\int_{\R^N}\mathbf{1}_{B_R}\left(|\nabla u|^{p-2}+\nabla(|\nabla u|^{p-2})+\int_{B_\ell}|z|^2\nu_{t,x}(dz)\right)
u\,\dd x\,\dd s\\
&\leq \|u\|_{\infty}\int_0^t\!\int_{B_R}\left(|\nabla u|^{p-2}+(p-2)|\nabla u|^{p-3}|D^2 u|+\int_{B_\ell}(1\wedge|z|^2)\,\nu_{t,x}(\dd z)\right)\dd x\,\dd s\\
&\leq M\left(|B_R|^{2/p}
\int_0^t\|\nabla u(s)\|_{L^p}^{p-2}\dd s+C_0^{1/2}
\Bigl(|B_R|^{4/p}\int_0^t\!\|\nabla u(s)\|_{L^p}^{p-4}\,\dd s
\Bigr)^{1/2}+t\,|B_R|\,\frac{2(2M)^{p-2}\omega_N}{sp(2-sp)}\right)\\
&<\infty,
\end{aligned}
\label{eq:a_bound}
\end{align}
where $M>0$ is a constant large enough and $\displaystyle C_0:=\int_0^t\int_{B_R}|\nabla u|^{p-2}|D^2u|^2\,\dd x\,\dd s<\infty$ according to Corollary \ref{cor:global_DiBe} in the Appendix. Moreover, since $|b|\le|\overline{b}|+|c_\ell|$, Lemma~\ref{lem:compensator} gives $\displaystyle\int_0^t\!\int_{\R^N}\mathbf{1}_{B_R}|c_\ell|\,u\,\dd x\,\dd s<\infty$, hence 1) holds.\\\\
$\bullet$ \textbf{Proof of 2):} From Theorem \ref{thm:existence_PDE}, since $\displaystyle\int_{\R^N} u(x)\dd x=1$, we get:
\begin{align}
\begin{aligned}
&\int_0^t\!\int_{\R^N}
\Bigl(\nu_{s,x}(B^c_{\ell\vee(|x|-R)})
+\mathbf{1}_{B_R}(x)\nu_{s,x}(B_\ell^c)\Bigr)
\mu_s(\dd x)\,\dd s\\
&\le2\int_0^t\!\int_{\R^N}\nu_{s,x}(B^c_{\ell})\mu_s(\dd x)\,\dd s=2\int_0^t\!\int_{\R^N}\left(\int_{|z|>\ell}
\frac{|u(s,x)-u(s,x+z)|^{p-2}}{|z|^{N+sp}}\dd z\right)\,u\,\dd x\,\dd s\\
&\le2(2M)^{p-2}\omega_N\int_0^t\!\int_{\R^N}\left(\int_{|z|>\ell}^{\infty}
\frac{1}{r^{1+sp}}\dd r\right)\,u\,\dd x\,\dd s=\frac{2(2M)^{p-2}\omega_N}{sp\,\ell^{sp}}\int_0^t\!\int_{\R^N}u\,\dd x\,\dd s\\
&\le\frac{2(2M)^{p-2}\omega_N}{sp\,\ell^{sp}}\cdot t<\infty.
\end{aligned}
\label{eq:cond2_done}
\end{align}
$\bullet$ \textbf{Proof of \eqref{eq:weak_FPE}:} Let $\varphi\in C^\infty_c(\R^N)$. Using the weak formulation \eqref{eq:weak_form}, the integration by parts
\[
-\int_{\R^N}|\nabla u|^{p-2}\nabla u\cdot\nabla\varphi\,\dd x=\int_{\R^N}\bigl(|\nabla u|^{p-2}\Delta\varphi+\nabla(|\nabla u|^{p-2})\cdot\nabla\varphi\bigr)u\,\dd x
\]
and the symmetrization identity \eqref{eq:sym}, we get
\begin{align*}
\frac{d}{dt}\int_{\R^N}\varphi\,u\,\dd x&=\int_{\R^N}|\nabla u|^{p-2}\Delta\varphi\,u\,\dd x
+\int_{\R^N}\nabla(|\nabla u|^{p-2})\cdot\nabla\varphi\,u\,\dd x\\
&+\int_{\R^N}\!\left(\int_{\R^N\setminus\{0\}}
\frac{|u(x)-u(x+z)|^{p-2}(\varphi(x+z)-\varphi(x))}
{|z|^{N+sp}}\dd z\right)\,u(x)\,\dd x\\
&=\int_{\R^N}\mathscr{L}_t\varphi(x)\,u(x)\,\dd x.
\end{align*}
Integrating the last equality from $0$ to $t$, and setting $\mu_t:=u(t,.)dx$, we deduce that
\begin{equation*}
\mu_t(f) = \mu_0(f)
+\int_0^t\mu_s(\mathcal{L}_s f)\,\dd s.
\end{equation*}
This completes the proof.
\end{proof}
\section{From FPKE \eqref{eq:FPE_general} to Martingale Problem \eqref{eq:MP}}
\label{sec:MP}
Let $\mathbb{D}:=D([0,\infty);\mathbb{R}^N)$ the space of functions $\omega:\mathbb{R}_{+}\rightarrow\mathbb{R}^N$ that are right-continuous with left limits, called càdlàg functions, and $X_s:\mathbb{D}\rightarrow\mathbb{R}^N$ the canonical process defined by 
\begin{align*}
    X_s(\omega)=\omega(s),\qquad s\in\mathbb{R}_{+}.
\end{align*}
Let $\phi\in C^2_c(\R^N)$ a test function and $x\in\mathbb{R}^N$. The \textit{Martingale Problem} associated to \eqref{eq:FPE_general} is the following
\begin{align}\label{eq:MP}
    \left\{\begin{aligned}
&\mathbb{E}_{\mathbb{Q}}\left[\phi(X_t)-\phi(X_s)-\int_s^t\mathcal{L}_r\phi(X_r)dr\,|\,\mathcal{F}_s\right]=0&&\\
        &\mathbb{Q}\circ X_0^{-1}=\mu_0,&&\mbox{ for all }0\le s\le t,
    \end{aligned}\right.
\end{align}
where $\mathcal{F}_s=\underset{q>s}{\bigcap}\mathcal{F}_q^0(\mathbb{D})$, $\mathcal{F}_q^0(\mathbb{D})=\sigma\left(X_r\,;\,0\le r\le q\right)$ being the smallest sigma-algebra with respect to which all the random variables $X_r$ ($0\le r\le q$) are measurable, and $\mathcal{L}_s$ is the general diffusion operator with jumps defined in \eqref{eq:L_assembly} with initial law $\mu_0$ at time $t=0$. 
\begin{remark} 
\begin{enumerate}
\item[a)] The family $(\mathcal{F}_s)_{s\ge 0}$ represents the information that accumulates over time. Here,
\[\mathcal{F}_t^0=\sigma(X_r,\ 0\le r\le t)\]
contains all the information provided by the trajectory up to time $t$. It is called a filtration since, for any $s\le t$, knowing the trajectory up to time $t$ necessarily provides at least as much information as knowing it only up to time $s$. Therefore,
\[\mathcal{F}_s^0 \subseteq \mathcal{F}_t^0,\]
 which implies that $(\mathcal{F}_s)_{s\ge 0}$ is an increasing family of sigma-algebras:
\[
\mathcal{F}_s \subseteq \mathcal{F}_t
\qquad \text{whenever } s\le t.
\]
\item[b)] Set
\begin{align}\label{M_t}
\displaystyle\mathcal{M}_t^{\phi}(X_t):=\phi(X_t)-\phi(X_0)-\int_0^t\mathcal{L}_s\phi(X_s)ds\quad\mbox{ for all }t\ge0.
\end{align}
Since $\displaystyle\mathcal{M}_s^{\phi}(X_s)$ is $\mathcal{F}_s$-measurable and integrable, then
\[\mathbb{E}\left[\mathcal{M}_s^{\phi}(X_s) \mid \mathcal{F}_s\right] = \mathcal{M}_s^{\phi}(X_s).\]
Consequently, 
\begin{align*}
0=\mathbb{E}_{\mathbb{Q}}\left[\phi(X_t)-\phi(X_s)-\int_s^t\mathcal{L}_r\phi(X_r)dr\,|\,\mathcal{F}_s\right]=\mathbb{E}_{\mathbb{Q}}\left[\mathcal{M}_t^{\phi}(X_t)-\mathcal{M}_s^{\phi}(X_s)\,|\,\mathcal{F}_s\right]
\end{align*} 
is equivalent to
\begin{align*}
\mathbb{E}_{\mathbb{Q}}\left[\mathcal{M}_t^{\phi}(X_t)\,|\,\mathcal{F}_s\right]=\mathcal{M}_s^{\phi}(X_s).
\end{align*}
In other words, knowing the entire history of the process up to time $s$, the best prediction of the future value $\mathcal{M}_t^{\phi}(X_t)$ is simply its current value $\mathcal{M}_s^{\phi}(X_s)$.
\item[c)] Thanks to the $\mathbb E_{\mathbb Q}\left[\mathbb E_{\mathbb Q}\left[Z|\;
\mathcal F_s
\right]\right]=\mathbb E_{\mathbb Q}\left[Z\right]$, starting from
\[
\mathbb E_{\mathbb Q}\left[
\phi(X_t)-\phi(X_s)
-\int_s^t \mathcal L_r\phi(X_r)\,dr
\;\middle|\;
\mathcal F_s
\right]=0.
\]
and taking the expectation on both sides gives
\[
\mathbb E_{\mathbb Q}[\phi(X_t)]
-
\mathbb E_{\mathbb Q}[\phi(X_s)]
=
\int_s^t
\mathbb E_{\mathbb Q}
\bigl[
\mathcal L_r\phi(X_r)
\bigr]
\,dr,
\]
called Dynkin's formula. Since $\mu_t=\mathbb Q\circ X_t^{-1}$, we have
\[
\mathbb E_{\mathbb Q}[\phi(X_t)]
=
\int_{\mathbb R^d}
\phi(x)\,\mu_t(dx)\quad\mbox{ and }\quad\mathbb E_{\mathbb Q}
\bigl[
\mathcal L_r\phi(X_r)
\bigr]
=
\int_{\mathbb R^d}
\mathcal L_r\phi(x)\,\mu_r(dx).
\]
Therefore,
\[
\int_{\mathbb R^d}\phi(x)\,\mu_t(dx)
-
\int_{\mathbb R^d}\phi(x)\,\mu_s(dx)
=
\int_s^t
\int_{\mathbb R^d}
\mathcal L_r\phi(x)\,\mu_r(dx)\,dr,
\]
which is the weak formulation of the forward Kolmogorov equation \eqref{eq:FPE_general}.
\end{enumerate}
\end{remark}
\begin{definition}[Martingale solution, see Definition~1.3 of \cite{RXZ2020}]
    A probability measure $\mathbb{Q}\in \mathcal{P}(\mathbb{D})$ is called a $\mathcal{F}_t$-\textit{martingale
solution} of \eqref{eq:MP} if $\mathbb{Q}$ solves \eqref{eq:MP}, that is 
\begin{enumerate}
    \item $\mathbb{Q}\circ X_0^{-1}=\mu_0$;
    \item For every $\phi\in C^2_c(\R^N)$ and any $x\in\mathbb{R}^N$,
\begin{align}\label{martingale}
\mathbb{E}_{\mathbb{Q}}\left[\phi(X_t)-\phi(X_s)-\int_s^t\mathcal{L}_r\phi(X_r)dr\,|\,\mathcal{F}_s\right]=0\quad\mbox{ for all }0\le s\le t.
\end{align}
\end{enumerate}
The process $\mathcal{M}_t^{\phi}(X_t)$ defined in \eqref{M_t} is called a $(\mathcal{F}_t)$-\textit{martingale} under $\mathbb{Q}$.
\label{eq:mart_def}
\end{definition}
\subsection{A superposition principle for non-local FPKE under integrability conditions}
Before proving the existence of a solution for \eqref{eq:MP}, it will be first needed to prove a new superposition principle for general FPKE of the type \eqref{eq:FPE_general} (without \eqref{particular_case}), under the following conditions
\begin{enumerate}
\item[(H1)] There exists a density function $u$ such that $\mu_t=u(t,.)dx$;
\item[(H2)] $\displaystyle\int_a^b\!\int_{\R^N}
\Bigl(\frac{|a_t(x)|}{1+|x|^2}+\frac{|b_t(x)|}{1+|x|}\Bigr)\,u(t,x)\dd x\,\dd t\le C(b-a)^{\alpha}$, for every $T>0$ and all $0\le a<b\le T$, for some $\alpha>0$ and $C=C(T)>0$;
\item[(H3)] $\displaystyle\sup_{t,x}\Bigl(
\frac{g_t^\nu(x)}{1+|x|^2}
+h_t^\nu(x)\Bigr)<\infty$,
\end{enumerate}
where $\displaystyle g_{\nu}(t,x):=\int_{B_\ell}|z|^2\nu_{t,x}(dz)$ and $\displaystyle h_t^\nu(x)
:=\int_{B_{\ell}^c}\log\left(1+\frac{|z|}{(1+|x|)}\right)\nu_{t,x}(\dd z)$.
\begin{theorem}[Superposition principle under integrability conditions]
\label{thm:15prime_weak}
Under conditions (H), for any weak solution
$(\mu_t)_{t\geq 0}$ of probability measures of the non-local FPKE \eqref{eq:FPE_general} in the sense of
Definition \ref{def:weak_FPE}, there exists
$\QQ\in\mathcal{M}_0^{\mu_0}(\mathcal{L}_t)$ such that \eqref{martingale} holds and
$\mu_t=\QQ\circ X_t^{-1}$ for all $t\geq 0$.
\end{theorem}
\begin{remark}
The novelty in Theorem \ref{thm:15prime_weak} lies essentially in assumptions (H) which replace the uniform growth condition (1.18) of \cite[Theorem 1.5]{RXZ2020} by an integrability condition on the local coefficients (see \cite[Remark~1.7]{RXZ2020} and the discussion in the Introduction). However its proof follows exactly the same steps. \begin{enumerate}
\item We start from a family of measures \((\mu_t)_{t \geq 0}\) which is a weak solution of the non-local Fokker--Planck equation associated with the operator \(L_t\). Since the coefficients of \(L_t\) may be merely measurable and have low regularity, one does not construct directly a process associated with \(L_t\). Instead, one first regularizes both the solution \(\mu_t\) and the coefficients of the operator. This yields a family of smooth measures \((\mu_t^\varepsilon)\) and a regularized operator \(L_t^\varepsilon\). By construction, \(\mu_t^\varepsilon\) solves the regularized Fokker--Planck equation
\[
\partial_t \mu_t^\varepsilon = (L_t^\varepsilon)^* \mu_t^\varepsilon .
\]
\item Since \(L_t^\varepsilon\) is now sufficiently regular, one can apply the result obtained in the regular case. Thus, one obtains a probability measure \(Q^\varepsilon\) on the space of c\`adl\`ag paths such that the canonical process solves the martingale problem associated with \(L_t^\varepsilon\), and whose time marginals are exactly \(\mu_t^\varepsilon\):
\[
Q^\varepsilon \circ X_t^{-1} = \mu_t^\varepsilon .
\]
In \cite[Lemma~3.3]{RXZ2020}, this step relies on the bound (1.18); under (H), it is given by Lemma \ref{lem:Qeps} (b) below.
\item Next, one proves that the family \((Q^\varepsilon)_\varepsilon\) is tight, which means that it does not escape to infinity in the path space. By Aldous' criterion (see \cite{JS1987}
or \cite[p.~356]{RXZ2020}), it suffices to check
the following two conditions:
\begin{itemize}
\item[(i)] For any $T > 0$, it holds that
\[
\lim_{N \to \infty} \sup_\varepsilon \QQ^\varepsilon
\left( \sup_{t \in [0,T]} |X_t| > N \right) = 0.
\]
\item[(ii)] For any $T, \delta_0 > 0$ and stopping time
$\tau < T - \delta_0$, it holds that
\[
\lim_{\delta \to 0} \sup_\varepsilon \sup_\tau
\QQ^\varepsilon \left( |X_{\tau+\delta} - X_\tau|
> \lambda \right) = 0,
\quad \forall \lambda > 0.
\]
\end{itemize}
\textbf{This is precisely where we intervene in the proof below, by proving that (ii) holds under assumptions (H).} Under (H), property (i) follows from Lemma \ref{lem:Qeps} (d) below.
\item Hence, one can extract a convergent subsequence, still denoted by \(Q^\varepsilon\), which converges weakly to a limiting probability measure \(Q\). The convergence of the marginals \(\mu_t^\varepsilon \to \mu_t\) then allows one to identify the time marginals of \(Q\). Namely, one obtains
\[
Q \circ X_t^{-1} = \mu_t .
\]
\item It remains to verify that \(Q\) indeed solves the martingale problem associated with the original operator \(L_t\), and not merely with an approximate limit of the operators \(L_t^\varepsilon\). To this end, one passes to the limit in the martingale identity satisfied by \(Q^\varepsilon\). The local terms are handled as in the proofs of Figalli \cite{Figalli2008} and Trevisan \cite{Trevisan2016}; the non-local term, however, requires a fine approximation of the L\'evy kernel. Once this passage to the limit has been carried out, one obtains that, for every test function \(f\),
\[
f(X_t)-f(X_0)-\int_0^t L_s f(X_s)\,ds
\]
is a martingale under \(Q\). In other words,
\[
Q \in M^{\mu_0}_0(L_t).
\]
Thus, the weak solution \((\mu_t)\) of the Fokker--Planck equation \eqref{eq:FPE_general} is represented as the family of time marginals of a solution to the martingale problem \eqref{eq:MP}.
\end{enumerate}
\end{remark}
\begin{lemma}[Regularized martingale solutions under (H)]
\label{lem:Qeps}
Assume (H), let $(\mu_t)_{t\ge0}$ be a weak solution of probability measures of \eqref{eq:FPE_general}, and let $\mu_t^\varepsilon$ and $(a^\varepsilon,b^\varepsilon,\nu^\varepsilon)$ be the regularized measures and coefficients of \cite[(3.2)--(3.5)]{RXZ2020}, $\varepsilon\in(0,\ell)$. Set
\[
G^\varepsilon(t,x):=\frac{|a_t^\varepsilon(x)|+g_t^{\nu^\varepsilon}(x)}{1+|x|^2}+\frac{|b_t^\varepsilon(x)|}{1+|x|}+h_t^{\nu^\varepsilon}(x),\qquad V(x):=\psi\bigl(\log(1+|x|^2)\bigr),
\]
where $\psi$ is the function given by \cite[Lemma~3.4]{RXZ2020}. Then, for every $T>0$:
\begin{enumerate}
\item[(a)] $\displaystyle\sup_{\varepsilon\in(0,\ell)}\int_0^T\!\!\int_{\R^N}G^\varepsilon\,\mu_t^\varepsilon\,\dd x\,\dd t<\infty$ and $\displaystyle\omega_T(K):=\sup_{\varepsilon\in(0,\ell)}\int_0^T\!\!\int_{\R^N}G^\varepsilon\,\mathbf{1}_{\{G^\varepsilon>K\}}\,\mu_t^\varepsilon\,\dd x\,\dd t\xrightarrow[K\to\infty]{}0$;
\item[(b)] for every $\varepsilon\in(0,\ell)$, there exists $\QQ^\varepsilon\in\mathcal{M}_0^{\mu_0^\varepsilon}(\mathcal{L}_t^\varepsilon)$ such that $\QQ^\varepsilon\circ X_t^{-1}=\mu_t^\varepsilon$ for all $t\ge0$;
\item[(c)] for every stopping time $\tau\le T$, the regular conditional distribution $\PP^\varepsilon_{\tau,X_\tau}$ of $\QQ^\varepsilon$ given $\mathcal{F}_\tau$ solves, for $\QQ^\varepsilon$-a.e. path, the martingale problem for $\mathcal{L}^\varepsilon_t$ starting from $X_\tau$ at time $\tau$; moreover, for every $\delta\in(0,1)$ and $K>0$,
\[
\EE_{\QQ^\varepsilon}\Bigl[\int_\tau^{\tau+\delta}G^\varepsilon(r,X_r)\,\dd r\Bigr]\le K\delta+\omega_{T+1}(K);
\]
\item[(d)] $\displaystyle\sup_{\varepsilon\in(0,\ell)}\EE_{\QQ^\varepsilon}\Bigl[\sup_{t\in[0,T]}V(X_t)^{1/2}\Bigr]<\infty$.
\end{enumerate}
\end{lemma}
\begin{proof} \mbox{} \\[1ex]
\noindent\textbf{(a).} We first collect the integrability framework (based on (H2)--(H3) and the de la Vall\'ee-Poussin theorem, see \cite[Chapter~II, Theorem~22]{DellacherieMeyer1978}) alongside the required probability weights:
\begin{itemize}
    \setlength{\itemsep}{1.5ex}
    \item[$\bullet$] $F(t,x) := \dfrac{|a_t(x)|+g_t^\nu(x)}{1+|x|^2} + \dfrac{|b_t(x)|}{1+|x|}$.
    \item[$\bullet$] There exists a convex, nondecreasing function $\Phi$ with $\Phi(0)=0$, $\displaystyle\lim_{r\to\infty} \frac{\Phi(r)}{r} = +\infty$, and
    \[
    \int_0^{T+1} \!\! \int_{\mathbb{R}^N} \Phi\bigl(4F(t,x)\bigr) \, \mu_t(\dd x) \, \dd t < \infty.
    \]
    \item[$\bullet$] $w_2(t,x) := \dfrac{\varepsilon\phi(x)}{\mu_t^\varepsilon(x)}$ and
    \[
    \pi^\varepsilon_{t,x}(\dd s,\dd y) := \frac{(1-\varepsilon)\rho_\varepsilon(t-s,x-y)\,\mu_s(\dd y)\,\dd s}{\mu_t^\varepsilon(x)},
    \]
    satisfying $\pi^\varepsilon_{t,x}(\mathbb{R}^{N+1}) + w_2(t,x) = 1$.
\end{itemize}

\smallskip

\noindent Using \cite[(3.1)]{RXZ2020}, the elementary bound $\log(1+2r) \le 2\log(1+r)$, and applying Jensen's inequality to the probability measure $\pi^\varepsilon_{t,x} + w_2(t,x)\delta_0$, we obtain:
\begin{align*}
G^\varepsilon(t,x) 
&\le \int_{\mathbb{R}^{N+1}} 2F(s,y) \, \pi^\varepsilon_{t,x}(\dd s,\dd y) + (\sqrt{N}+1) w_2(t,x) + 2\sup_{s,y} h_s^\nu(y), \\[3ex]
\Phi\bigl(G^\varepsilon(t,x)\bigr) 
&\le \frac{1}{2} \int_{\mathbb{R}^{N+1}} \Phi\bigl(4F(s,y)\bigr) \, \pi^\varepsilon_{t,x}(\dd s,\dd y) + \frac{1}{2} w_2(t,x) \Phi\bigl(2(\sqrt{N}+1)\bigr) + \frac{1}{2} \Phi\Bigl(4\sup_{s,y} h_s^\nu(y)\Bigr).
\end{align*}

\noindent Multiplying by $\mu_t^\varepsilon(x)$ and integrating over $(0,T) \times \mathbb{R}^N$, property (H3) yields the uniform bound:
\[
\sup_{\varepsilon \in (0,\ell)} \int_0^T \!\! \int_{\mathbb{R}^N} \Phi(G^\varepsilon) \, \mu_t^\varepsilon \, \dd x \, \dd t < \infty, \qquad \text{which yields conclusion (a) since } \lim_{r \to \infty} \frac{\Phi(r)}{r} = +\infty.
\]
\bigskip

\noindent\textbf{(b).} The growth condition of \cite[Theorem~2.6]{RXZ2020} relies on uniform bounds unavailable under (H). Instead, we apply the martingale problem formulation from \cite[Theorem~1.11 and Corollary~1.12]{KurtzStockbridge2001} (with $B \equiv 0$ and a one-point control space). We define the functional space, the generator extension, and the bounding function as follows:
\begin{align*}
    \mathcal{D} &:= C_c^\infty(\mathbb{R}^N) \oplus \mathbb{R} \quad \text{with} \quad \mathcal{L}_t^\varepsilon 1 := 0, \\[1.5ex]
    \Psi^\varepsilon(t,x) &:= 1 + \dfrac{|a_t^\varepsilon(x)|+g_t^{\nu^\varepsilon}(x)}{1+|x|^2} + \dfrac{|b_t^\varepsilon(x)|}{1+|x|}.
\end{align*}

\noindent We now check the hypotheses of \cite[Corollary~1.12]{KurtzStockbridge2001}.

\smallskip

\noindent\textit{Continuity.} By (H2), $a^\varepsilon$ and $b^\varepsilon$ are continuous on $\mathbb{R}_+\times\mathbb{R}^N$, and $a_t^\varepsilon(x)$ is invertible since $a_t^\varepsilon \ge \varepsilon\phi/\mu_t^\varepsilon\, I$. Moreover, by (H3),
\[\sup_{(s,y)\in[0,n]\times B_n} \int_{\mathbb{R}^N} (1\wedge\vert{}z\vert{}^2)\,\nu_{s,y}(\dd z) < \infty \qquad \text{for every } n\in\mathbb{N},\]
so that the estimate \cite[(3.9)]{RXZ2020} holds, ensuring $\mathcal{L}_t^\varepsilon f$ and $\Psi^\varepsilon$ are continuous.

\smallskip

\noindent\textit{Growth.} If $\operatorname{spt} f \subset B_R$, then we have the bound:
\[
|\mathcal{L}_t^\varepsilon f| \le c_f \Psi^\varepsilon, \qquad \text{with } c_f < \infty \text{ by (H3)}.
\]

\smallskip

\noindent\textit{Integrability.} By result (a), for every $t>0$,
\[\int_0^t \!\! \int_{\mathbb{R}^N} \Psi^\varepsilon \mu_s^\varepsilon \, \dd x \, \dd s \le t + \int_0^t \!\! \int_{\mathbb{R}^N} G^\varepsilon \mu_s^\varepsilon \, \dd x \, \dd s < \infty.\]

\smallskip

\noindent\textit{Forward equation.} By \cite[Proposition~3.2 (iii)]{RXZ2020},
\[\int_{\mathbb{R}^N} f \, \dd\mu_t^\varepsilon = \int_{\mathbb{R}^N} f \, \dd\mu_0^\varepsilon + \int_0^t \!\! \int_{\mathbb{R}^N} \mathcal{L}_s^\varepsilon f \, \dd\mu_s^\varepsilon \, \dd s, \qquad f \in \mathcal{D},\ t\ge0.\]

\smallskip

\noindent\textit{Structure.} $\mathcal{D}$ is a separable algebra which separates the points of $\mathbb{R}^N$. For fixed $(t,x)$,
\[
\mathcal{L}_t^\varepsilon f(x) = \lim_{n\to\infty} n\bigl(\mathbb{E}[f(x+Z^{t,x}_{1/n})]-f(x)\bigr),
\]
where $Z^{t,x}$ is the L\'evy process with triplet $\bigl(b_t^\varepsilon(x), \, 2a_t^\varepsilon(x), \, \nu^\varepsilon_{t,x}\bigr)$. This representation ensures that $\mathcal{L}_t^\varepsilon$ satisfies \cite[(1.10)]{KurtzStockbridge2001} and the positive maximum principle, hence it is a pre-generator in the sense of \cite{KurtzStockbridge2001}.

\medskip

\noindent By \cite[Corollary~1.12]{KurtzStockbridge2001}, there exists a measurable process $X^\varepsilon$ on some $(\Omega,\mathcal{F},(\mathcal{F}_t),\mathbb{P})$ satisfying $\mathbb{P}\circ(X_t^\varepsilon)^{-1}=\mu_t^\varepsilon$, such that for every $f\in\mathcal{D}$, the following is an $(\mathcal{F}_t)$-martingale:
\[f(X_t^\varepsilon) - f(X_0^\varepsilon) - \int_0^t \mathcal{L}_s^\varepsilon f(X_s^\varepsilon) \, \dd s.\]

\medskip

\noindent\textbf{C\`adl\`ag modification.} Let $\widehat{\mathbb{R}}^N := \mathbb{R}^N \cup \{\infty\}$ be the one-point compactification of $\mathbb{R}^N$, and extend every $f\in C_c^\infty(\mathbb{R}^N)$ by $f(\infty) := 0$. Since our generator is unbounded, we explicitly construct a c\`adl\`ag modification $Z$ of $X^\varepsilon$ following the classical regularization argument \cite[Chapter~4, Section~3]{EthierKurtz1986} and \cite[Remark~3.5]{KurtzStockbridge2001}:

\smallskip

\noindent\textit{Topological embedding.} Choose a countable family $\{g_k\}_{k\ge1} \subset C_c^\infty(\mathbb{R}^N)$ which separates the points of $\mathbb{R}^N$ and such that, for every $x\in\mathbb{R}^N$, $g_k(x)\neq0$ for some $k$ (e.g.\ smooth bumps centred at rational points with rational radii). Then the following map is a homeomorphism from $\mathbb{R}^N$ onto its image in the compact product space:
\[
G(x) := \bigl(g_k(x)\bigr)_{k\ge1} \in \prod_k \bigl[ -\|g_k\|_\infty, \|g_k\|_\infty \bigr].
\]
Let $\mathbf{0}:=(0,0,\dots)$ denote the zero sequence. Since each $g_k$ has compact support, $G(x)\to\mathbf{0}$ as $|x|\to\infty$, and $\mathbf{0}\notin G(\mathbb{R}^N)$ by the choice of $\{g_k\}$. Hence the closure of $G(\mathbb{R}^N)$ is $G(\mathbb{R}^N) \cup \{\mathbf{0}\}$, which is identified with $\widehat{\mathbb{R}}^N$ by sending $\infty$ to $\mathbf{0}$.

\smallskip

\noindent\textit{Rational limits.} For each $k$, we define the associated process:
\[
M^{k}_t := g_k(X_t^\varepsilon) - g_k(X_0^\varepsilon) - \int_0^t \mathcal{L}_s^\varepsilon g_k(X_s^\varepsilon) \, \dd s.
\]
This process is an integrable martingale, which is directly guaranteed by the bounds:
\[
|\mathcal{L}_s^\varepsilon g_k| \le c_{g_k}\Psi^\varepsilon \qquad \text{and} \qquad \mathbb{E} \Biggl[ \int_0^T \Psi^\varepsilon(s,X^\varepsilon_s) \, \dd s \Biggr] < \infty.
\]
Since the integral term is a.s. continuous, Doob's regularity theorem ensures that $M^k$ (and consequently $g_k(X^\varepsilon_q)$ for $q \in \mathbb{Q}_+$) admits a.s. right and left limits along the rationals.

\smallskip

\noindent\textit{Stochastic continuity.} Since $\mathcal{D}$ is an algebra, $g_k^2 \in \mathcal{D}$. For any $h > 0$, the $L^2$ increment directly vanishes:
\begin{align*}
\mathbb{E}\Bigl[ \bigl|g_k(X^\varepsilon_{t+h}) - g_k(X^\varepsilon_t)\bigr|^2 \Bigr]
&= \int_t^{t+h} \!\! \int_{\mathbb{R}^N} \mathcal{L}_s^\varepsilon(g_k^2) \, \dd\mu_s^\varepsilon \, \dd s - 2\,\mathbb{E} \Biggl[ g_k(X^\varepsilon_t) \int_t^{t+h} \mathcal{L}_s^\varepsilon g_k(X^\varepsilon_s) \, \dd s \Biggr]\xrightarrow[h\to 0]{} 0.
\end{align*}

\noindent Consequently, the limit $Z_t := \lim_{q\downarrow t,\,q\in\mathbb{Q}} G(X^\varepsilon_q)$ exists for all $t \ge 0$ outside a null set. This defines a c\`adl\`ag process in $\widehat{\mathbb{R}}^N$ such that $Z_t = X^\varepsilon_t$ a.s. for each fixed $t$ (hence $Z_t \in \mathbb{R}^N$ a.s. and $\mathbb{P} \circ Z_t^{-1} = \mu^\varepsilon_t$). Extending the generator via $\mathcal{L}_t^\varepsilon f(\infty) := 0$, the process $f(Z_t) - f(Z_0) - \displaystyle\int_0^t \mathcal{L}_s^\varepsilon f(Z_s) \, \dd s$ is a c\`adl\`ag martingale for every $f \in \mathcal{D}$. However, the process $Z$ may a priori still reach the boundary point $\infty$ at random times.

\medskip

\noindent To prove that $Z$ never reaches $\infty$, we introduce a cutoff function $\zeta \in C_c^\infty(B_2)$ such that $0 \le \zeta \le 1$ and $\zeta \equiv 1$ on $B_1$. This naturally defines the rescaled test functions and the auxiliary process:
\begin{align*}
    f_n &:= \zeta(\cdot/n) \quad \text{for } n \ge 2, \\[1ex]
    Y_t^n &:= 1 - f_n(Z_t) \in [0,1].
\end{align*}

\noindent The process $Y^n$ decomposes into an integral drift and a c\`adl\`ag martingale $M^n$:
\[Y_t^n = Y_0^n - \int_0^t \mathcal{L}_s^\varepsilon f_n(Z_s) \, \dd s - M_t^n, \qquad\mbox{with}\qquad \mathbb{E}\bigl[\vert{}M_T^n\vert{}\bigr] \le \mathbb{E}[Y_T^n] + \mathbb{E}[Y_0^n] + \mathbb{E} \Biggl[ \int_0^T \vert{}\mathcal{L}_s^\varepsilon f_n(Z_s)\vert{} \, \dd s \Biggr].\]

\smallskip

\noindent Using the pointwise bound $1-f_n \le \mathbf{1}_{B_n^c}$, Doob's maximal inequality and Markov's inequality directly yield:
\begin{align*}
\mathbb{P}\Bigl(\sup_{t\in[0,T]} Y_t^n \ge 1\Bigr)
&\le 6\,\mu_0^\varepsilon(B_n^c) + 3\,\mu_T^\varepsilon(B_n^c) + 6\int_0^T \!\! \int_{\mathbb{R}^N} |\mathcal{L}_s^\varepsilon f_n| \, \mu_s^\varepsilon \, \dd x \, \dd s.
\end{align*}

\noindent To bound the generator applied to $f_n$, we collect the pointwise estimates for its components:
\begin{align*}
    \textit{Derivatives:} \quad & \nabla f_n \text{ and } \nabla^2 f_n \text{ vanish outside } \{n \le |x| \le 2n\} \text{ and are bounded by } \dfrac{\|\nabla\zeta\|_\infty}{n}, \dfrac{\|\nabla^2\zeta\|_\infty}{n^2}. \\[2ex]
    \textit{Small jumps:} \quad & \Theta_{f_n}(x;z) = 0 \text{ unless } n-\ell < |x| < 2n+\ell. \\[2ex]
    \textit{Big jumps:} \quad & \text{For } |z| > \ell, \text{ we have } |f_n(x+z)-f_n(x)| \le C_\zeta \log\biggl(1+\dfrac{|z|}{1+|x|}\biggr).
\end{align*}

\noindent This bound for the big jumps follows by distinguishing three cases. If $|x|\le n/2$, the integrand vanishes unless $|z|>n/2$, which yields $|z|/(1+|x|)\ge 1/2$. If $|x|>3n$, it vanishes unless $|z|>|x|/3$, which yields $|z|/(1+|x|)\ge 1/4$. Finally, if $n/2<|x|\le 3n$, we use $1+|x|\le 4n$ and $\min\{1,r\}\le \log(1+r)/\log 2$.

\smallskip

\noindent These estimates guarantee the uniform bound $|\mathcal{L}_s^\varepsilon f_n| \le C_\zeta G^\varepsilon$ for all $n \ge 2$. Furthermore, $\mathcal{L}_s^\varepsilon f_n(x) \to 0$ for every $(s,x)$: the local terms and the small jumps vanish for $n > |x|+\ell$, and the big jumps tend to $0$ by dominated convergence, since $f_n \to 1$ and $\nu^\varepsilon_{s,x}(B_\ell^c) < \infty$.

\medskip

\noindent Recall that $Y_t^n=1$ (resp. $Y_{t-}^n=1$) whenever $Z_t=\infty$ (resp. $Z_{t-}=\infty$). Passing to the limit via the dominated convergence theorem directly yields:
\begin{align*}
    \lim_{n\to\infty} \mathbb{P} \Biggl( \displaystyle\sup_{t\in[0,T]} Y_t^n \ge 1 \Biggr) &= 0, \\[3ex]
    \mathbb{P} \Bigl( \exists t \le T : Z_t = \infty \text{ or } Z_{t-} = \infty \Bigr) &= 0.
\end{align*}
\noindent Thus $\mathbb{Q}^\varepsilon := \mathbb{P} \circ Z^{-1}$ is supported on $\mathbb{R}^N$ and verifies conclusion (b) via approximation in $C^2_c(\mathbb{R}^N)$.
\bigskip

\noindent\textbf{(c).} The first assertion extends the regular conditional probability argument of \cite[Theorem~6.1.3, p.~140]{StroockVaradhan1979}. Although originally stated for bounded coefficients and continuous paths, its proof solely relies on \cite[Theorem~1.2.10, p.~28]{StroockVaradhan1979} tested on a countable dense subset of $C_c^\infty(\mathbb{R}^N)$, followed by an approximation. The required assumptions hold here:
\begin{itemize}
    \setlength{\itemsep}{1.5ex}
    \item[$\bullet$] \textbf{Topological properties:} $\mathbb{D}$ is Polish (ensuring the existence of the conditional distribution) and its canonical filtration is countably generated.
    \item[$\bullet$] \textbf{Martingale integrability:} For each test function $f$, the associated process is an integrable $\mathbb{Q}^\varepsilon$-martingale. Indeed, using result (a):
    \[\vert{}\mathcal{L}_r^\varepsilon f\vert{} \le c_f \Psi^\varepsilon \qquad \text{and} \qquad \mathbb{E}_{\mathbb{Q}^\varepsilon} \int_0^T \Psi^\varepsilon(r,X_r) \, \dd r < \infty.\]
    \item[$\bullet$] \textbf{Approximation step ($f_n \to f$ in $C^2$):} The limit holds via dominated convergence, justified by $\vert{}\mathcal{L}_r^\varepsilon(f_n-f)\vert{} \le C\Vert{}f_n-f\Vert{}_{C^2}\Psi^\varepsilon$ and the finiteness of the inner expectation for $\mathbb{Q}^\varepsilon$-a.e. path:
    \[\mathbb{E}_{\mathbb{Q}^\varepsilon} \Biggl[ \mathbb{E}_{\mathbb{P}^\varepsilon_{\tau,X_\tau}} \int_\tau^T \Psi^\varepsilon(r,X_r) \, \dd r \Biggr] = \mathbb{E}_{\mathbb{Q}^\varepsilon} \int_\tau^T \Psi^\varepsilon(r,X_r) \, \dd r < \infty.\]
\end{itemize}

\smallskip

\noindent For the expectation estimate, we split $G^\varepsilon$ at the threshold $K$ and use $\mathbb{Q}^\varepsilon \circ X_r^{-1} = \mu_r^\varepsilon$:
\begin{align*}
\mathbb{E}_{\mathbb{Q}^\varepsilon} \Biggl[ \int_\tau^{\tau+\delta} G^\varepsilon(r,X_r) \, \dd r \Biggr]
&\le \mathbb{E}_{\mathbb{Q}^\varepsilon} \Biggl[ \int_\tau^{\tau+\delta} K \, \dd r \Biggr] + \int_0^{T+1} \!\! \int_{\mathbb{R}^N} G^\varepsilon \mathbf{1}_{\{G^\varepsilon > K\}} \, \mu_r^\varepsilon \, \dd x \, \dd r \\[2ex]
&\le K\delta + \omega_{T+1}(K).
\end{align*}

\bigskip

\noindent\textbf{(d).} Applying \cite[(2.2), Lemma~3.5]{RXZ2020} alongside the elementary bound $\dfrac{\vert{}x\vert{}}{1+\vert{}x\vert{}^2} \le \dfrac{2}{1+\vert{}x\vert{}}$, we obtain the pointwise estimate $\mathcal{L}_t^\varepsilon V \le 16\,G^\varepsilon$. By the martingale property of $\mathbb{Q}^\varepsilon$ (see \cite[Remark~1.4]{RXZ2020}, whose proof only uses the local boundedness of the coefficients, which holds by \cite[(3.3)--(3.5)]{RXZ2020}, $\mu_t^\varepsilon\ge\varepsilon\phi$ and (H)), we get:
\[
V(X_t) \le V(X_0) + 16 \int_0^t G^\varepsilon(r,X_r) \, \dd r + M_t^\varepsilon, \qquad \text{(where } M^\varepsilon \text{ is a c\`adl\`ag local martingale)}.
\]

\noindent Applying the stochastic Gronwall inequality \cite[Lemma~2.4]{RXZ2020} and the transfer theorem, we conclude:
\[\mathbb{E}_{\mathbb{Q}^\varepsilon} \Biggl[ \sup_{t\in[0,T]} V(X_t)^{1/2} \Biggr] \le C \Biggl( \sup_{\varepsilon\in(0,\ell)} \int_{\mathbb{R}^N} V \, \dd\mu_0^\varepsilon + 16 \sup_{\varepsilon\in(0,\ell)} \int_0^T \!\! \int_{\mathbb{R}^N} G^\varepsilon \mu_r^\varepsilon \, \dd x \, \dd r \Biggr)^{1/2} < \infty.\]

\end{proof}
\begin{proof}[Proof of Theorem \ref{thm:15prime_weak}]As explained in the earlier remark, we only need to prove the validity of property (ii) in Aldous' criterion under assumptions (H) following \cite[Lemma~3.6]{RXZ2020}, the existence of $\QQ^\varepsilon$ and property (i) being given by Lemma \ref{lem:Qeps} (b) and (d).\par\bigskip\noindent
\noindent By Lemma~\ref{lem:Qeps} (d), we have the following uniform bound:
\begin{enumerate}
\item[1.] $\displaystyle\mathbb{E}_{\mathbb{Q}^{\varepsilon}} \Biggl[ \sup_{0\le s\le T} V(X_s)^{1/2} \Biggr] \le C$,
\end{enumerate}
where $V(x) := \psi\bigl(\log(1+|x|^2)\bigr)$. 

\medskip

\noindent Next, we introduce the shifted function $V_y(x) := \psi\bigl(\log(1+|x-y|^2)\bigr)$ and observe its basic properties:
\begin{itemize}
    \setlength{\itemsep}{1.5ex}
    \item[$\bullet$] $V(x) = V_0(x)$ and $V_y(y) = 0$.
    \item[$\bullet$] For any threshold $a > 0$, the condition $|Z_s-y| > a$ ensures that:
    \[
    V_y(Z_s)^{1/2} > \psi\bigl(\log(1+a^2)\bigr)^{1/2}.
    \]
\end{itemize}
\noindent Hence,
\begin{enumerate}
\item[2.] $\displaystyle\mathbb{Q}^{\varepsilon}\left(|Z_s-y|>a\right)\le\mathbb{Q}^{\varepsilon}
\left(V_y(Z_s)^{1/2}>\psi\!\left(\log(1+a^2)\right)^{1/2}\right)$.
\end{enumerate}
Also, from Markov's inequality and Lemma \ref{lem:Qeps} (c), we have
\begin{enumerate}
\item[3.] $\mathbb{Q}^{\varepsilon}\left(V(Z)^{1/2}>\psi\!\left(\log(1+a^2)\right)^{1/2}\right)\le\frac{\mathbb{E}_{\mathbb{Q}^{\varepsilon}}
\left[V(Z)^{1/2}\right]}{\psi\!\left(\log(1+a^2)\right)^{1/2}}$,
\item[4.] $\mathbb{Q}^{\varepsilon}
\bigl(|X_{\tau+\delta}-X_{\tau}|>\lambda\bigr)=\mathbb{E}_{\mathbb{Q}^{\varepsilon}}
\left[\mathbb{P}^{\varepsilon}_{\tau,X_{\tau}}
\bigl(|X_{\tau+\delta}-X_{\tau}|>\lambda\bigr)\right]$,
\end{enumerate}
where $\mathbb{P}^{\varepsilon}_{\tau,X_\tau}$ denotes the regular conditional distribution of $\QQ^\varepsilon$ given $\mathcal{F}_\tau$, which solves the martingale problem for $\mathcal{L}^\varepsilon_t$ starting from $X_\tau$ at time $\tau$ (Lemma \ref{lem:Qeps} (c)). Set $s:=\tau$, $y:=X_s$ and $\mathbb{P}_{s,y}^{\varepsilon}:=\mathbb{P}^{\varepsilon}_{\tau,X_\tau}$. Since $V_{y}(X_s)=0$, from Ito's formula, \cite[Lemma 2.3]{RXZ2020} inequality (2.2), \cite[Lemma 3.5]{RXZ2020} and the stochastic Gronwall inequality \cite[Lemma 2.4]{RXZ2020}, we check that
\begin{enumerate}
\item[a)]$\displaystyle V_y(X_t)= V_y(X_s)+ \int_s^t\mathcal{L}_r^\varepsilon V_y(X_r)\,\dd r+ M_t^\varepsilon=\int_s^t\mathcal{L}_r^\varepsilon V_y(X_r)\,\dd r+ M_t^\varepsilon$,
\item[b)] $\displaystyle\mathcal{L}_t^\varepsilon V_y(x)
\le 2 \left(
\frac{|a_t^\varepsilon(x)|
+ \langle x-y, b_t^\varepsilon(x) \rangle^+
+ g_t^{\nu^\varepsilon}(x)}{1+|x-y|^2}
+ 2H_t^{\nu^\varepsilon}(x,y) \right)$,
\item[c)] $H_t^{\nu^\varepsilon}(x,y)\le2(1+|y|)h_t^{\nu^\varepsilon}(x)$, hence, using $1+|x|^2\le2(1+|x-y|^2)(1+|y|^2)$, $1+|x|\le(1+|x-y|)(1+|y|)$ and $\frac{|x-y|}{1+|x-y|^2}\le\frac{2}{1+|x-y|}$,
\[
\bigl(\mathcal{L}_t^\varepsilon V_y(x)\bigr)^+\le16(1+|y|)^2G^\varepsilon(t,x),
\]
where $G^\varepsilon$ is defined in Lemma \ref{lem:Qeps},
\item[d)] $\displaystyle\EE_{\PP_{s,y}^\varepsilon}\left[\left(\sup_{q\in[s,t]}V_y(X_q)\right)^{1/2}\right]\le \left(C\EE_{\PP_{s,y}^\varepsilon}\left[\sup_{q\in[s,t]}\int_s^q\bigl(\mathcal{L}_r^\varepsilon V_y(X_r)\bigr)^+dr\right]\right)^{1/2}$
\item[]\hspace{4.4cm}$\displaystyle\le \left(C\EE_{\PP_{s,y}^\varepsilon}\left[\int_s^t\bigl(\mathcal{L}_r^\varepsilon V_y(X_r)\bigr)^+dr\right]\right)^{1/2}$,
\end{enumerate}
where $M_t^\varepsilon$ is a local Martingale and $\displaystyle H_t^{\nu^{\varepsilon}}(x,y)
:=\int_{B_{\ell}^c}\log\left(1+\frac{|z|}{1+|x-y|}\right)\nu_{t,x}^{\varepsilon}(\dd z)$.
Consequently, by c) and d),
\begin{align}\label{new}
\EE_{\PP_{s,y}^\varepsilon}\left[\sup_{q\in[s,s+\delta]}V_y(X_{q})^{1/2}\right]\le 4\sqrt{C}\,(1+|y|)\left(\EE_{\PP_{s,y}^\varepsilon}\left[\int_s^{s+\delta}G^\varepsilon(r,X_r)\,\dd r\right]\right)^{1/2}.
\end{align}
Therefore, regrouping 1.-4. and \eqref{new}, for every stopping time $\tau$, every $\delta>0$, every $\lambda>0$, and every $R>0$,
\begin{align*}
\begin{aligned}
\mathbb{Q}^{\varepsilon}\bigl(|X_{\tau+\delta}-X_{\tau}|>\lambda\bigr)
&=\mathbb{E}_{\mathbb{Q}^{\varepsilon}}\left[\mathbb{P}^{\varepsilon}_{\tau,X_{\tau}}\bigl(|X_{\tau+\delta}-X_{\tau}|>\lambda\bigr)\right]
\\
&\le\mathbb{Q}^{\varepsilon}\bigl(\mathbf{1}_{\{|X_{\tau}|> R\}}\bigr)+\mathbb{E}_{\mathbb{Q}^{\varepsilon}}\left[\mathbf{1}_{\{|X_{\tau}|\le R\}}\,\mathbb{P}^{\varepsilon}_{\tau,X_{\tau}}\bigl(|X_{\tau+\delta}-X_{\tau}|>\lambda\bigr)\right]
\\
&\le\mathbb{Q}^{\varepsilon}\left(|X_\tau|>R\right)+\mathbb{E}_{\mathbb{Q}^{\varepsilon}}\left[\mathbf{1}_{\{|X_{\tau}|\le R\}}\,\mathbb{P}^{\varepsilon}_{\tau,X_{\tau}}\bigl(|X_{\tau+\delta}-X_{\tau}|>\lambda\bigr)\right]
\\
&\le\mathbb{Q}^{\varepsilon}\left(V(X_\tau)^{1/2}>\psi\!\left(\log(1+R^2)\right)^{1/2}
\right)\\
&+\mathbb{E}_{\mathbb{Q}^{\varepsilon}}\left[\mathbf{1}_{\{|X_{\tau}|\le R\}}\,\mathbb{P}^{\varepsilon}_{\tau,X_{\tau}}\left(V_{X_\tau}(X_{\tau+\delta})^{1/2}>\psi\!\left(\log(1+\lambda^2)\right)^{1/2}\right)\right]\\
&\le\frac{\mathbb{E}_{\mathbb{Q}^{\varepsilon}}\left[V(X_\tau)^{1/2}\right]}{\psi\!\left(\log(1+R^2)\right)^{1/2}}+\mathbb{E}_{\mathbb{Q}^{\varepsilon}}\left[\mathbf{1}_{\{|X_{\tau}|\le R\}}\,\frac{\mathbb{E}_{\mathbb{P}^{\varepsilon}_{\tau,X_{\tau}}}\left[V_{X_{\tau}}(X_{\tau+\delta})^{1/2}\right]}{\psi\!\left(\log(1+\lambda^2)\right)^{1/2}}\right]
\\
&\le\frac{\mathbb{E}_{\mathbb{Q}^{\varepsilon}}\left[\sup_{s\in[0,T]}V(X_s)^{1/2}\right]}{\psi\!\left(\log(1+R^2)\right)^{1/2}}+\mathbb{E}_{\mathbb{Q}^{\varepsilon}}\left[\mathbf{1}_{\{|X_{\tau}|\le R\}}\,\frac{\mathbb{E}_{\mathbb{P}^{\varepsilon}_{\tau,X_{\tau}}}\left[\sup_{q\in[\tau,\tau+\delta]}V_{X_{\tau}}(X_{q})^{1/2}\right]}{\psi\!\left(\log(1+\lambda^2)\right)^{1/2}}\right]
\\
&\le
\frac{C}{\psi\!\left(\log(1+R^2)\right)^{1/2}}
+\frac{4\sqrt{C}(1+R)\,\EE_{\QQ^\varepsilon}\left[\left(\EE_{\PP^\varepsilon_{\tau,X_\tau}}\left[\int_\tau^{\tau+\delta}G^\varepsilon(r,X_r)\,\dd r\right]\right)^{1/2}\right]}{\psi\!\left(\log(1+\lambda^2)\right)^{1/2}}\\
&\le
\frac{C}{\psi\!\left(\log(1+R^2)\right)^{1/2}}
+\frac{4\sqrt{C}(1+R)\bigl(K\delta+\omega_{T+1}(K)\bigr)^{1/2}}{\psi\!\left(\log(1+\lambda^2)\right)^{1/2}},
\end{aligned}
\end{align*}
where we used \eqref{new}, Jensen's inequality, $\EE_{\QQ^\varepsilon}\bigl[\EE_{\PP^\varepsilon_{\tau,X_\tau}}[\,\cdot\,]\bigr]=\EE_{\QQ^\varepsilon}[\,\cdot\,]$ and Lemma \ref{lem:Qeps} (c), for every $K>0$. Letting $\delta \to 0$ first, then $K\to\infty$ and finally $R \to \infty$,
one sees that (ii) is satisfied. This completes the proof.
\end{proof}
\subsection{Existence of a solution for the Martingale Problem (MP)}
\begin{theorem}[Existence of a martingale solution]Under assumptions (A)+(B), if $\mu_t:=u(t,\cdot)\dd x$, where $u$ is the solution of \eqref{eq:PDE} given by Theorem~\ref{thm:existence_PDE}, then $\mu_t$
is a solution of \eqref{eq:FPE_general} (with \eqref{particular_case}) in the sense of Definition~\ref{def:weak_FPE} and there exists $\QQ\in\mathcal{M}_0^{\mu_0}(\mathcal{L}_t)$ solution of \eqref{eq:MP} in the sense of Definition~\ref{eq:mart_def} such that $\mathbb{Q}\circ X_t^{-1}=\mu_t$.
\label{th:4.5}
\end{theorem}
\begin{proof} It suffices to verify assumptions (H) for $a_t,b_t,\nu_t$ defined in \eqref{particular_case} and then apply Theorem \ref{thm:15prime_weak}.\\\\
\textbf{Proof of (H1).} It is a direct consequence of Theorem \ref{thm:bridge_complete}.\par\bigskip\noindent
\textbf{Proof of (H2).} Let $T>0$ and $0\le a<b\le T$. Since $b=\overline{b}+c_\ell$ and $|\nabla(|\nabla u|^{p-2})|\le(p-2)|\nabla u|^{p-3}|D^2u|$, we get by Corollary \ref{cor:global_DiBe}
\begin{align*}
&\int_a^b\!\int_{\R^N}
\Bigl(\frac{|a_t(x)|}{1+|x|^2}+\frac{|b_t(x)|}{1+|x|}\Bigr)\,u(t,x)\dd x\,\dd t\\
&\le\int_a^b\!\int_{\R^N}|\nabla u|^{p-2}\,u\,\dd x\,\dd t+(p-2)\int_a^b\!\int_{\R^N}\bigl(|\nabla u|^{\frac{p-2}{2}}|D^2u|\bigr)
|\nabla u|^{\frac{p-4}{2}}u\,\dd x\,\dd t+\int_a^b\!\int_{\R^N}|c_\ell|\,u\,\dd x\,\dd t\\
&\le\left(\int_a^b\!\int_{\R^N}|u|^{\frac{p}{2}}\,\dd x\,\dd t\right)^{\frac{2}{p}}\left(\int_a^b\!\int_{\R^N}|\nabla u|^{p}\,\dd x\,\dd t\right)^{\frac{p-2}{p}}\\
&+(p-2)\left(\int_a^b\!\int_{\R^N}|\nabla u|^{p-2}|D^2u|^2\,\dd x\,\dd t\right)^{1/2}\left(\int_a^b\!\int_{\R^N}|u|^{\frac{p}{2}}\,\dd x\,\dd t\right)^{\frac{2}{p}}\left(\int_a^b\!\int_{\R^N}|\nabla u|^{p}\,\dd x\,\dd t\right)^{\frac{p-4}{2p}}\\
&+C\,\|u\|_\infty^{\frac{p-2}{p}}\|\nabla u\|_{L^p((0,T)\times\R^N)}^{p-2}(b-a)^{\frac2p}\\
&\le(b-a)^{2/p}\|u\|_{\infty}^{\frac{p-2}{p}}\left(\int_{\R^N}|u_0|dx\right)^{2/p}\left((1+C)\|\nabla u\|_{L^p((0,T)\times\R^N)}^{p-2}+(p-2)C_0^{1/2}\|\nabla u\|_{L^p((0,T)\times\R^N)}^{\frac{p-4}{2}}\right)\\
&\le C_T(b-a)^{2/p},
\end{align*}
\noindent where we used Lemma~\ref{lem:compensator} for the term in $c_\ell$, alongside the following bounds:
\begin{itemize}
    \setlength{\itemsep}{2ex}
    \item[$\bullet$] $|u|^{\frac{p}{2}} \le \|u\|_\infty^{\frac{p}{2}-1} |u|$,
    \item[$\bullet$] $\displaystyle\int_{\mathbb{R}^N} u(t) \, \dd x \le \int_{\mathbb{R}^N} |u_0| \, \dd x \le 1$,
    \item[$\bullet$] $C_0 := \displaystyle\int_0^T \!\! \int_{\mathbb{R}^N} |\nabla u|^{p-2} |D^2u|^2 \, \dd x \, \dd t < \infty$.
\end{itemize}

\smallskip

\noindent This establishes (H2) with $\alpha = \frac{2}{p}$. Notice that no information on the support of $u$ is needed.\\\
\noindent\textbf{Proof of (H3).} Since $0\leq u\leq M$, we have $|u(t,x)-u(t,x+z)|\le M$. Therefore, using $\log(1+r)\leq r$ and $1+|x|\ge1$,
\begin{align*}
\frac{g_t^\nu(x)}{1+|x|^2}+h_t^\nu(x)&\le\int_{|z|\le\ell}|z|^2\nu_{t,x}(\dd z)+\int_{|z|>\ell}
|z|\,\nu_{t,x}(\dd z)\\
&\leq M^{p-2}\omega_N\left(\int_0^\ell
r^{1-sp}\dd r+\int_\ell^\infty
r^{-sp}\dd r\right)
\\
&=M^{p-2}\omega_N\left(\frac{\ell^{2-sp}}{2-sp}+\frac{\ell^{1-sp}}{sp-1}\right)
<\infty,
\end{align*}
since $1<sp<2$. This implies
\[
\sup_{t,x}\left(\frac{g_t^\nu(x)}{1+|x|^2}+h_t^\nu(x)\right)<\infty.
\]
\end{proof}
\section{From Martingale Problem \eqref{eq:MP} to SDE \eqref{eq:SDE}}
\label{sec:SDE}
\subsection{Existence of a solution for the Stochastic Differential Equation \eqref{eq:SDE}}
\noindent The SDE we consider is
\begin{equation}\label{eq:SDE}
\left\{
\begin{aligned}
\dd X(t)
&= \nabla\bigl(|\nabla u(t,X(t))|^{p-2}\bigr)\dd t
+\sqrt{2}\,|\nabla u(t,X(t))|^{\frac{p-2}{2}}\dd W(t)\\
&\quad +\int_{\R^N\setminus\{0\}}z\,
\mathcal{N}(\dd t,\dd z),
\quad t>0,\\
\Law_{X(t)}&=u(t,x)\dd x,\quad t\ge0.
\end{aligned}
\right.
\end{equation} 
where $W$ is a Brownian motion and $\mathcal{N}$ is an integer-valued random measure
with compensator $\nu_{t,X(t^-)}(\dd z)\dd t$.
\begin{definition}[Solution-process/weak solution,
see \cite{JS1987} Definition~II.2.24]
\label{def:sol_process}\quad
\begin{enumerate}
\item A \emph{solution-process}
(or strong solution) to \eqref{eq:JS_SDE}
on a basis $\mathcal{B}'$ relative to driving terms
$(W,p)$ is a càdlàg adapted process $Y$ such that
for each $i\leq N$:
\begin{equation}
Y^i = \xi^i
+\beta^i(Y)\cdot t
+\sum_{j\leq N}\gamma^{ij}(Y_-)\cdot W^j
+h^i\circ\delta(Y_-)\star(p-q)
+{h'}^i\circ\delta(Y_-)\star p,
\label{eq:sol_process}
\end{equation}
where:
\begin{align*}
g \cdot A_t:= \int_0^t g(Y_{r^-}) \mathrm{d}A_r\quad\mbox{ and }\quad(g \star \mu)_t:= \int_0^t \!\int g(z) \mu(\mathrm{d}r, \mathrm{d}z).
\end{align*}
\item A \emph{weak solution}
(or measure-valued solution) to \eqref{eq:JS_SDE}
with initial condition $\eta$ is a probability measure
$\PP$ on the canonical space $(\Omega,\mathcal{F})$
such that there exist:
\begin{itemize}
\item a stochastic basis
$\mathcal{B}'=(\Omega',\mathcal{F}',\PP';
(\mathcal{F}'_t)_{t\geq 0})$
with driving terms $(W,p)$,
\item an $\mathcal{F}'_0$-measurable variable $\xi$
with $\mathcal{L}(\xi)=\eta$,
\item a solution-process $Y$ on $\mathcal{B}'$
in the sense of \textup{a)},
\end{itemize}
such that $\PP = \PP'\circ Y^{-1}$
(i.e.\ $\PP$ is the law of $Y$).
\end{enumerate}
\end{definition}

\begin{definition}[Local characteristics of a semimartingale,
see \cite{JS1987} Definition~II.2.6]
\label{def:char}
Let $X$ be a semimartingale.
The \emph{local characteristics} $(B,C,\nu)$ of $X$ are:
\begin{itemize}
    \item $B$ is the predictable drift process (finite variation, $B_0=0$),
    \item $C = \langle X^c, X^c \rangle$ is the quadratic variation of the continuous martingale part $X^c$,
    \item $\nu$ is the predictable compensator of the jump measure
    \[
        \mu^X(\omega, \mathrm{d}t, \mathrm{d}z) := \sum_{s > 0} \mathbf{1}_{\{\Delta X_s \neq 0\}} \, \delta_{(s, \Delta X_s(\omega))}(\mathrm{d}t, \mathrm{d}z).
    \]
\end{itemize}
\end{definition}
\begin{remark}
The random measure $\mu^X$ counts the jumps of $X$: it places a unit mass $\delta_{(s, \Delta X_s)}$ at $(s, \Delta X_s)$ whenever $\Delta X_s \neq 0$. Its compensator $\nu$ is the unique predictable random measure such that $\mu^X - \nu$ is a local martingale measure.
\end{remark}

\begin{theorem}[Existence of a weak solution of the SDE]
\label{thm:5.4}
Under assumptions (A)+(B), there exists a weak solution $X=(X(t))_{t\geq 0}$ of the
McKean--Vlasov SDE \eqref{eq:SDE} in the sense
of Definition~\ref{def:sol_process} such that
$u(t,x)\dd x=\Law_{X(t)}$ for all $t\geq 0$.
\end{theorem}
\begin{proof} \mbox{} \\[1ex]
By Theorem~\ref{th:4.5}, there exists a probability measure $\mathbb{P} \in \mathcal{P}(\mathbb{D})$ such that for every $f \in C^2_c(\mathbb{R}^N)$, the following process is a martingale under $\mathbb{P}$:
\begin{equation}
f(X(t)) - f(X(0)) - \int_0^t \mathcal{L}_r f(X(r)) \, \dd r.
\label{eq:RXZ_mart}
\end{equation}

\medskip

\noindent This property extends to $f \in C^2_b(\mathbb{R}^N)$ via the approximation:
\[
f_n(x) := f(x)\chi(x/n), \qquad \text{where } \chi \in C^\infty_c(\mathbb{R}^N) \text{ and } \chi \equiv 1 \text{ on the unit ball } B_1 := \{|x|<1\}.
\]

\noindent The passage to the limit is justified by the dominated convergence theorem, relying on the bound:
\[
|\mathcal{L}_r(f_n)| \le C \|f\|_{C^2} \bigl( |\overline{a}| + |\overline{b}| + |c_\ell| + g_\nu + \nu_{r,x}(B^c_\ell) \bigr).
\]
This majorant is integrable with respect to $u(r,x) \, \dd x \, \dd r$ as shown in the proof of (H2)--(H3).

\medskip

\noindent Therefore, according to \cite[Theorem~II.2.42]{JS1987}, the process $X$ admits the characteristics $(B,C,\nu)$ in the sense of Definition~\ref{def:char}, explicitly given by:
\begin{align*}
B &:= \int_0^t \bigl(\overline{b} + c_\ell\bigr)(r,X_r) \, \dd r, \\[2ex]
C &:= \int_0^t 2\,\overline{a}(r,X_r)\,\mathbb{I}_N \, \dd r, \\[2ex]
\nu &:= \nu_{t,X_{t^-}}(\dd z) \, \dd t.
\end{align*}

\medskip

\noindent Next, we represent the jump measure $\nu_{t,x}$ as the pushforward of the reference measure:
\[
F(\dd z,\dd\theta) := |z|^{-N-sp} \, \dd z \, \dd\theta \qquad \text{on the space } E := \mathbb{R}^N \times (0,\infty).
\]

\smallskip

\noindent This pushforward is constructed via the following mapping, restricted to $z \in \mathbb{R}^N \setminus \{0\}$:
\[
(z,\theta) \mapsto \delta\bigl(t,x,(z,\theta)\bigr) := z \mathbf{1}_{\bigl\{ \theta \le |u(t,x)-u(t,x+z)|^{p-2} \bigr\}}.
\]

\medskip

\noindent According to \cite[Theorem~III.2.26]{JS1987}, we can construct a stochastic basis $\bigl(\Omega', \mathcal{F}', (\mathcal{F}'_t)_{t \ge 0}, \mathbb{P}_{\mathrm{hyb}}\bigr)$ supporting the following objects:
\begin{itemize}
    \setlength{\itemsep}{1.5ex}
    \item[$\bullet$] A Brownian motion $W$.
    \item[$\bullet$] A Poisson random measure $p(\dd t, \dd e)$ on $\mathbb{R}_+ \times E$ with intensity $q(\dd t,\dd e) = \dd t \, F(\dd e)$.
    \item[$\bullet$] An adapted c\`adl\`ag process $\Xh$.
\end{itemize}

\smallskip

\noindent This process is constructed such that its marginal laws coincide with our PDE solution:
\begin{equation}
\mathbb{P}_{\mathrm{hyb}} \circ \bigl({\Xh}(t)\bigr)^{-1} = \mathbb{P} \circ X_t^{-1} = \mu_t = u(t, \cdot) \, \dd x, \quad \forall t \ge 0,
\label{eq:law_final}
\end{equation}
and $Y=\Xh$ is a solution to the nonlinear McKean--Vlasov SDE with jumps:
\begin{equation}
\begin{cases}
    Y_0 = \xi, \\[2ex]
    \dd Y_t = \beta(t,Y_t) \, \dd t + \gamma(t,Y_{t-}) \, \dd W_t \\
    \qquad\quad + \displaystyle\int_E h\bigl(\delta(t,Y_{t-},e)\bigr) \, (p-q)(\dd t,\dd e) + \displaystyle\int_E h'\bigl(\delta(t,Y_{t-},e)\bigr) \, p(\dd t,\dd e).
\end{cases}
\label{eq:JS_SDE}
\end{equation}
where $\xi$ is an $\mathcal{F}'_0$-measurable random variable with law $\mu_0$.

\medskip

\noindent The coefficients for this SDE are grouped and defined as follows:
\begin{equation}
\left\{
\begin{aligned}
h(z) &:= z \mathbf{1}_{\{|z|\le\ell\}}, \\[1.5ex]
h'(z) &:= z - h(z), \\[1.5ex]
\beta(t,x) &:= \overline{b}(t,x) + c_\ell(t,x) \in \mathbb{R}^N, \\[1.5ex]
\gamma(t,x) &:= \sqrt{2\,\overline{a}(t,x)} \, \mathbb{I}_N \in \mathbb{M}_{N\times N}(\mathbb{R}), \\[1.5ex]
\delta\bigl(t,x,(z,\theta)\bigr) &:= z \mathbf{1}_{\{\theta \le |u(t,x)-u(t,x+z)|^{p-2}\}} \in \mathbb{R}^N.
\end{aligned}
\right.
\label{eq:our_coeff}
\end{equation}

\medskip

\noindent Finally, we introduce the integer-valued random measure:
\[\mathcal{N}(\dd t,A) := \displaystyle\int_E \mathbf{1}_A\bigl(\delta(t,Y_{t-},e)\bigr) \, p(\dd t,\dd e), \qquad A \subset \mathbb{R}^N \setminus \{0\},\]
whose compensator is given by $\nu_{t,Y_{t-}}(\dd z) \, \dd t$. Since the small jumps are summable by Lemma~\ref{lem:compensator}, we can recombine the jump integrals:
\[\displaystyle\int_E h(\delta) \, (p-q) + \displaystyle\int_E h'(\delta) \, p = \displaystyle\int_{\mathbb{R}^N} z \, \mathcal{N}(\dd t,\dd z) - c_\ell(t,Y_{t-}) \, \dd t.\]

\smallskip

\noindent Substituting this relation back into the system, equation \eqref{eq:JS_SDE} with the coefficients \eqref{eq:our_coeff} reduces to \eqref{eq:SDE}. This completes the proof.
\end{proof}
\begin{remark}[Nonlinear L\'evy--Khintchine formula]
\label{thm:LK}
Theorem \ref{thm:5.4} gives a nonlinear version of the L\'evy--Khintchine formula for every $\phi\in C^2_c(\R^N)$, which is
\begin{equation}
\widetilde{L}\phi(x)
= b_u(t,x)\cdot\nabla\phi(x)
+\frac{1}{2}\mathrm{tr}
\bigl(A_u(t,x)\nabla^2\phi(x)\bigr)
+\int_{\R^N\setminus\{0\}}
\Theta_\phi(x;z)\,\nu_u(t,x,\dd z),
\label{eq:LK_nonlinear}
\end{equation}
where $(b_u,A_u,\nu_u)\equiv(\overline{b}+c_\ell,\,2\overline{a}\,\mathbb{I}_N,\,\nu_{t,x})$. Also, according to \cite[Theorem 2.34]{JS1987}, the following canonical representation holds
\begin{align*}
&X_{hyb}(t)=X_{hyb}(0)+B+X^c_{hyb}+h\star(\mu^{X_{hyb}}-\nu)+(x-h(x))\star\mu^{X_{hyb}},\\
&\mbox{where }X^c_{hyb}(t)=\int_0^t\sqrt{2\overline{a}(r,X_{hyb}(r))}\,\dd W_r.
\end{align*}
\end{remark}

\section{Appendix}
\subsection{Weighted second-order regularity}
\noindent We prove the analogue of \cite[Proposition~3.1 chapter 8]{DiBenedetto1993} for equation of the type \eqref{eq:PDE}.
\begin{proposition}[Weighted second-order regularity,
hybrid equation, $p\geq 4$]
\label{prop:hybrid31}
Let $p\geq 4$, $s\in(0,1)$ with $sp>1$, $N\geq 1$.
Let $u$ be a local weak solution of
\begin{equation}
u_t - \operatorname{div}(|\nabla u|^{p-2}\nabla u)
-\Delta_p^s u = 0
\quad\text{in }\mathcal{O}_T := \mathcal{O}\times(0,T),
\label{eq:hybrid}
\end{equation}
with $u\in C_{\mathrm{loc}}(0,T;L^2_{\mathrm{loc}}(\mathcal{O}))
\cap L^p(0,T;W^{1,p}_{\mathrm{loc}}(\mathcal{O}))
\cap L^p(0,T;W^{s,p}_{\mathrm{loc}}(\mathcal{O}))$, and moreover $u\in L^\infty((0,T)\times\R^N)\cap L^p(0,T;W^{1,p}(\R^N))\cap L^p(0,T;W^{s,p}(\R^N))$ (which is the case for the solution of \eqref{eq:PDE} given by Theorem~\ref{thm:existence_PDE}).
Then for every $\sigma\in(0,1)$ and every cylinder
$Q(\theta,\rho) = K_\rho\times(-\theta,0)\Subset\mathcal{O}_T$,
there exists $\gamma = \gamma(p,N)>0$ such that:

\begin{align*}
&\iint_{Q(\theta,\rho)}
|\nabla u|^{p-2}|\nabla(\partial_{x_j}u)|^2\zeta^2
\,\dd x\,\dd\tau
\notag\\
&\leq\,\gamma\left[
\iint_{Q(\theta,\rho)}\!
\left(
\frac{|\nabla u|^2}{(1-\sigma)\theta}
+\frac{|\nabla u|^p}{(1-\sigma)^2\rho^2}
\right)\dd x\,\dd\tau\right.
\notag\\
&\left.\quad
+\iint_{Q(\theta,\rho)}\!\int_{\R^N}
\frac{|u(x,\tau)-u(y,\tau)|^{p-2}
\bigl(|\partial_{x_j}u(x,\tau)|^2+|\partial_{x_j}u(y,\tau)|^2\bigr)
|\zeta(x)-\zeta(y)|^2}
{|x-y|^{N+sp}}
\,\dd y\,\dd x\,\dd\tau
\right],
\end{align*}

\end{proposition}
\begin{proof}
The proof follows DiBenedetto \cite{DiBenedetto1993}, Section~3,
with the nonlocal term $\Delta_p^s u$ treated on the same
footing as the local $p$-Laplacian.\\\\
In what follows, we use the following notation:
 \begin{enumerate}
 \item $K_\rho := (-\rho,\rho)^N
= \{x\in\R^N : |x_i|<\rho,\;i=1,\ldots,N\}
$, $\rho>0$ fixed;
\item $K_\rho^{|\eta|}
:=\{x\in K_\rho:\operatorname{dist}(x,\partial K_\rho)>|\eta|\}$;
\item $\displaystyle[g]_h(.) := h^{-1}\int_t^{t+h}g(\cdot,\tau)\dd\tau$, $h>0$ (Steklov average of $g$);
\item $\displaystyle\delta_j g(x,t)
:= \frac{g(x+\eta e_j,t)-g(x,t)}{\eta}$,
 $x\in K_\rho^{|\eta|}$ (discrete derivative);
 \end{enumerate}
\noindent
Taking the Steklov average of \eqref{eq:PDE}
over $(t,t+h)$ and applying $\delta_j$ gives
\begin{equation}
[\delta_j u]_{h,t}
- \operatorname{div}([\delta_j(|\nabla u|^{p-2}\nabla u)]_h)
- [\delta_j(\Delta_p^s u)]_h = 0, \mbox{ a.e. in }K_\rho\times(0,T-h).
\label{eq:diff}
\end{equation}
Multiplying \eqref{eq:diff} by $\varphi = [\delta_j u]_h\zeta^2$, where $\zeta=\zeta(x,t)$ is a cutoff function with $\zeta(\cdot,t)\in C_c^{\infty}(K_\rho)$ and $\zeta(\cdot,-\theta)=0$,
and integrating over $K_\rho^{|\eta|}\times(-\theta,t)$
for $-\theta<t\leq 0$.
\begin{align}\label{Steklov}
\begin{aligned}
\int_{-\theta}^t\!\int_{K_\rho}
[\delta_j u]_{h,t}\,[\delta_j u]_h\zeta^2\,\dd x\,\dd\tau&=-\int_{-\theta}^t\!\int_{K_\rho}
[\delta_j(|\nabla u|^{p-2}\nabla u)]_h
\cdot\nabla\left([\delta_j u]_h\zeta^2\right)\,\dd x\,\dd\tau\\
&+\int_{-\theta}^t\!\int_{K_\rho}
[\delta_j(\Delta_p^s u)]_h(x)\,[\delta_j u]_h\zeta^2
\,\dd x\,\dd\tau.
\end{aligned}
\end{align}
Since $\displaystyle\partial_t\left(\frac{1}{2}|[\delta_j u]_h|^2\right)
= [\delta_j u]_{h,t}\,[\delta_j u]_h$
and $\zeta(\cdot,-\theta)=0$:
\begin{align}
\begin{aligned}
\int_{-\theta}^t\!\int_{K_\rho}
[\delta_j u]_{h,t}\,[\delta_j u]_h\,\zeta^2\,\dd x\,\dd\tau
&= \int_{-\theta}^t\!\partial_t\left(\frac{1}{2}\int_{K_\rho}
|[\delta_j u]_h(\tau)|^2\zeta^2\,\dd x\right)\,\dd\tau
- \int_{-\theta}^t\!\int_{K_\rho}
|\delta_j u_h|^2\,\zeta\,\zeta_t\,\dd x\,\dd\tau\\
&= \int_{K_\rho}
\frac{|[\delta_j u]_h(t)|^2}{2}\zeta^2\,\dd x
- \int_{-\theta}^t\!\int_{K_\rho}
|\delta_j u_h|^2\,\zeta\,\zeta_t\,\dd x\,\dd\tau.
\end{aligned}
\label{eq:time}
\end{align}
The Steklov averages satisfy $[g]_h\to g$ in $L^p_{\mathrm{loc}}$
as $h\to 0$ for any $g\in L^p_{\mathrm{loc}}$.
Since $|\nabla u|\in L^p_{\mathrm{loc}}(\mathcal{O}_T)$ and
$u\in C_{\mathrm{loc}}(0,T;L^2_{\mathrm{loc}})$,
all limits as $h\to 0$ are justified by standard arguments
(see \cite[Proposition~3.1 chapter 8]{DiBenedetto1993}). Therefore, using \eqref{eq:time} in \eqref{Steklov} and letting $h\rightarrow0$ while $|\eta|>0$ remains fixed, we get
\begin{align}\label{DiBene0}
\begin{aligned}
\int_{K_\rho}
\frac{|\delta_j u(t)|^2}{2}\zeta^2\,\dd x&=\int_{-\theta}^t\!\int_{K_\rho}
|\delta_j u_h|^2\,\zeta\,\zeta_t\,\dd x\,\dd\tau-\int_{-\theta}^t\!\int_{K_\rho}
\delta_j(|\nabla u|^{p-2}\nabla u)
\cdot\nabla\left(\delta_j u\,\zeta^2\right)\,\dd x\,\dd\tau\\
&+\int_{-\theta}^t\!\int_{K_\rho}
\delta_j(\Delta_p^s u)(x)\,(\delta_j u)\zeta^2
\,\dd x\,\dd\tau.
\end{aligned}
\end{align}
$\bullet$ \textbf{Step 1:} By the fundamental theorem of calculus and Young's inequality, we have
\begin{align}
\begin{aligned}
&-\delta_j(|\nabla u|^{p-2}\nabla u)(x)\cdot\nabla(\delta_j u\,\zeta^2)\\
&=-\left(\frac{1}{\eta}\int_0^1
\frac{\dd}{\dd\sigma}
\left[\left|\sigma\nabla u(x+\eta e_j)+(1-\sigma)\nabla u(x)\right|^{p-2}
\left(\sigma\nabla u(x+\eta e_j)+(1-\sigma)\nabla u(x)\right)\right]\dd\sigma\right)\\
&\qquad\times\left(\nabla(\delta_j u)\,\zeta^2+2(\delta_j u)\zeta\nabla\zeta\right)\\
&=-\mathbb{A}_\eta\nabla(\delta_j u)\cdot\left(\nabla(\delta_j u)\,\zeta^2+2(\delta_j u)\zeta\nabla\zeta\right)\\
&\le\left(\int_0^1\left|\nabla^{(\sigma,\eta)}u\right|^{p-2}\dd\sigma\right)\left(-\left|\nabla(\delta_j u)\right|^2\,\zeta^2+\frac{1}{2}\left|\nabla(\delta_j u)\right|^2\,\zeta^2+2(p-1)^2\left|\delta_j u\right|^2\left|\nabla\zeta\right|^2\right)\\
&\le\left(\int_0^1\left|\nabla^{(\sigma,\eta)}u\right|^{p-2}\dd\sigma\right)\left(-\frac{1}{2}\left|\nabla(\delta_j u)\right|^2\,\zeta^2+2(p-1)^2\left|\delta_j u\right|^2\left|\nabla\zeta\right|^2\right),
\end{aligned}
\label{eq:local_expand}
\end{align}
where $\nabla^{(\sigma,\eta)}u:= \sigma\nabla u(x+\eta e_j)+(1-\sigma)\nabla u(x)$ and $\mathbb{A}_\eta:=\int_0^1D\mathcal{A}(\nabla^{(\sigma,\eta)}u)\dd\sigma$, $\mathcal{A}(\xi):=|\xi|^{p-2}\xi$; we used that the matrix $D\mathcal{A}(\xi)=|\xi|^{p-2}\bigl(\mathbb{I}_N+(p-2)\frac{\xi\otimes\xi}{|\xi|^2}\bigr)$ satisfies $D\mathcal{A}(\xi)w\cdot w\ge|\xi|^{p-2}|w|^2$ and $|D\mathcal{A}(\xi)w|\le(p-1)|\xi|^{p-2}|w|$. Consequently, 
\begin{align}
\begin{aligned}
&-\int_{-\theta}^t\!\int_{K_\rho}
\delta_j(|\nabla u|^{p-2}\nabla u)
\cdot\nabla(\delta_j u\,\zeta^2)\,\dd x\,\dd\tau\\
&\leq\int_{-\theta}^t\!\int_{K_\rho}
\left(\int_0^1\left|\nabla^{(\sigma,\eta)}u\right|^{p-2}\dd\sigma\right)\left(-\frac{1}{2}\left|\nabla(\delta_j u)\right|^2\,\zeta^2+2(p-1)^2\left|\delta_j u\right|^2\left|\nabla\zeta\right|^2\right)\,\dd x\,\dd\tau.
\end{aligned}
\label{eq:Young_local}
\end{align}
$\bullet$ \textbf{Step 2:} By the change of variables $z\mapsto z+\eta e_j$ (translation invariance of the Lebesgue measure),
\begin{equation}
\begin{split}
\int_{\R^N}
\frac{|u(.)-u(z)|^{p-2}(u(.)-u(z))}
{|.-z|^{N+sp}}\,\dd z
&=\int_{\R^N}
\frac{|u(.)-u(z+\eta e_j)|^{p-2}\bigl(u(.)-u(z+\eta e_j)\bigr)}
{|.-(z+\eta e_j)|^{N+sp}}\,\dd z.
\end{split}
\label{eq:change_var}
\end{equation}
Therefore, we obtain
\begin{align}\label{sept}
\begin{aligned}
\delta_j(\Delta_p^s u)(x)
&= -\frac{1}{\eta}\PV\!\int_{\R^N}
\left(
\frac{|u(x+\eta e_j)-u(y)|^{p-2}
(u(x+\eta e_j)-u(y))}{|x+\eta e_j-y|^{N+sp}}\right.\\
&\left.- \frac{|u(x)-u(y)|^{p-2}(u(x)-u(y))}{|x-y|^{N+sp}}
\right)dy\\
&=-\frac{1}{\eta}\PV\!\int_{\R^N}
\left(
|u(x+\eta e_j)-u(y+\eta e_j)|^{p-2}
(u(x+\eta e_j)-u(y+\eta e_j))\right.\\
&\left.- |u(x)-u(y)|^{p-2}(u(x)-u(y))
\right)
\frac{\dd y}{|x-y|^{N+sp}}.
\end{aligned}
\end{align}Also, by the fractional integration by part formula, we check that
\begin{align*}
&\int_{-\theta}^t\!\int_{K_\rho}
\delta_j(\Delta_p^s u)(x)\,(\delta_j u)\zeta^2
\,\dd x\,\dd\tau=-\frac12\int_{-\theta}^t\!\iint_{\R^N\times\R^N}
\frac{\delta_jA_pu(x,y)}{|x-y|^{N+sp}}\left(\delta_j u(x)\zeta^2(x)-\delta_j u(y)\zeta^2(y)\right)dy\,\dd x\,\dd\tau,
\end{align*}
where $\displaystyle A_pu(x,y):=\left|u(x)-u(y)\right|^{p-2}\left(u(x)-u(y)\right)$ and $\delta_jA_pu(x,y)=\frac{1}{\eta}\bigl(A_pu(x+\eta e_j,y+\eta e_j)-A_pu(x,y)\bigr)$.
Now, set $\displaystyle Du(x,y):=u(x)-u(y)$ and notice that $$\left(\delta_j u(x)-\delta_j u(y)\right)=\frac{1}{\eta}\left(Du(x+\eta e_j,y+\eta e_j)-Du(x,y)\right).$$ Proceeding as previously, we prove that
\begin{align*}
&-\delta_jA_p(x,y)\left(\delta_j u(x)\zeta^2(x)-\delta_j u(y)\zeta^2(y)\right)\\
&= -\left\{\frac{1}{\eta}\int_0^1
\frac{\dd}{\dd\sigma}
\left[\left|\sigma Du(x+\eta e_j,y+\eta e_j)+(1-\sigma)Du(x,y)\right|^{p-2}
\left(\sigma Du(x+\eta e_j,y+\eta e_j)\right.\right.\right.\\
&\left.\left.\left.+(1-\sigma)Du(x,y)\right)\right]\dd\sigma\right\}\left\{\left(\delta_j u(x)-\delta_j u(y)\right)\frac{\zeta^2(x)+\zeta^2(y)}{2}+\frac{\delta_j u(x)+\delta_j u(y)}{2}\left(\zeta^2(x)-\zeta^2(y)\right)\right\}\\
&=-\left(\delta_j u(x)-\delta_j u(y)\right)(p-1)\left(\int_0^1\left|\nabla^{(\sigma,\eta,y)}u\right|^{p-2}\dd\sigma\right)\\
&\qquad\times\left\{\left(\delta_j u(x)-\delta_j u(y)\right)\frac{\zeta^2(x)+\zeta^2(y)}{2}+\frac{\delta_j u(x)+\delta_j u(y)}{2}\left(\zeta(x)+\zeta(y)\right)\left(\zeta(x)-\zeta(y)\right)\right\}\\
&\le(p-1)\left(\int_0^1\left|\nabla^{(\sigma,\eta,y)}u\right|^{p-2}\dd\sigma\right)\left(-\left|\delta_j u(x)-\delta_j u(y)\right|^2\,\frac{\zeta^2(x)+\zeta^2(y)}{2}\right.\\
&\qquad\left.+8^{-1}\left|\delta_j u(x)-\delta_j u(y)\right|^2\,\left(\zeta(x)+\zeta(y)\right)^2+2^{-1}\left|\delta_j u(x)+\delta_j u(y)\right|^2\left|\zeta(x)-\zeta(y)\right|^2\right)\\
&\le(p-1)\left(\int_0^1\left|\nabla^{(\sigma,\eta,y)}u\right|^{p-2}\dd\sigma\right)\left(-\left|\delta_j u(x)-\delta_j u(y)\right|^2\,\frac{\zeta^2(x)+\zeta^2(y)}{2}\right.\\
&\qquad\left.+4^{-1}\left|\delta_j u(x)-\delta_j u(y)\right|^2\,\left(\zeta^2(x)+\zeta^2(y)\right)+2^{-1}\left|\delta_j u(x)+\delta_j u(y)\right|^2\left|\zeta(x)-\zeta(y)\right|^2\right)\\
&=(p-1)\left(\int_0^1\left|\nabla^{(\sigma,\eta,y)}u\right|^{p-2}\dd\sigma\right)\left(-\frac{1}{2}\left|\delta_j u(x)-\delta_j u(y)\right|^2\,\frac{\zeta^2(x)+\zeta^2(y)}{2}\right.\\
&\qquad\left.+2^{-1}\left|\delta_j u(x)+\delta_j u(y)\right|^2\left|\zeta(x)-\zeta(y)\right|^2\right)
\end{align*}
where $\nabla^{(\sigma,\eta,y)}u:=\sigma Du(x+\eta e_j,y+\eta e_j)+(1-\sigma)Du(x,y)$. Consequently,
\begin{align}\label{DiBene}
\begin{aligned}
&\int_{-\theta}^t\!\int_{K_\rho}
\delta_j(\Delta_p^s u)(x)\,(\delta_j u)\zeta^2
\,\dd x\,\dd\tau\le(p-1)\int_{-\theta}^t\!\int_{K_\rho}\int_{\R^N}\frac{1}{|x-y|^{N+sp}}\left(\int_0^1\left|\nabla^{(\sigma,\eta,y)}u\right|^{p-2}\dd\sigma\right)\\
&\times\left(-\frac{1}{8}\left|\delta_j u(x)-\delta_j u(y)\right|^2\,\frac{\zeta^2(x)+\zeta^2(y)}{2}+2^{-1}\left|\delta_j u(x)+\delta_j u(y)\right|^2\left|\zeta(x)-\zeta(y)\right|^2\right)\,dy\,\dd x\,\dd\tau.
\end{aligned}
\end{align}
The last inequality follows by symmetry in $(x,y)$, since the positive part of the integrand vanishes unless $x \in K_\rho$ or $y \in K_\rho$ (the factor $\frac12$ on the negative part gives the coefficient $-\frac18$).\\\\
$\bullet$ \textbf{Step 3:} Assembling \eqref{DiBene0},\eqref{eq:Young_local} and \eqref{DiBene}, rearranging the terms
and taking the supremum over $-\theta<t<0$ gives
\begin{align}
\begin{aligned}
&\sup_{-\theta<t<0}\int_{K_\rho}
\frac{|\delta_j u(t)|^2}{2}\zeta^2\,\dd x +\frac{1}{2}\int_{Q(\theta,\rho)}
\Bigl(\int_0^1|\nabla^{(\sigma,\eta)}u|^{p-2}\dd\sigma\Bigr)
|\nabla(\delta_j u)|^2\zeta^2\,\dd x\,\dd\tau\\
&+\frac{p-1}{8}\int_{-\theta}^t\!\int_{K_\rho}\int_{\R^N}\frac{1}{|x-y|^{N+sp}}\left(\int_0^1\left|\nabla^{(\sigma,\eta,y)}u\right|^{p-2}\dd\sigma\right)\left|\delta_j u(x)-\delta_j u(y)\right|^2\,\frac{\zeta^2(x)+\zeta^2(y)}{2}\,dy\,\dd x\,\dd\tau\\
&\leq \int_{Q(\theta,\rho)}
|\delta_j u|^2|\zeta_t|\zeta\,\dd x\,\dd\tau+2(p-1)^2\int_{Q(\theta,\rho)}
\Bigl(\int_0^1|\nabla^{(\sigma,\eta)}u|^{p-2}\dd\sigma\Bigr)
|\delta_j u|^2|\nabla\zeta|^2\,\dd x\,\dd\tau\\
&+2^{-1}(p-1)\int_{-\theta}^t\!\int_{K_\rho}\int_{\R^N}\frac{1}{|x-y|^{N+sp}}\left(\int_0^1\left|\nabla^{(\sigma,\eta,y)}u\right|^{p-2}\dd\sigma\right)\left|\delta_j u(x)+\delta_j u(y)\right|^2\left|\zeta(x)-\zeta(y)\right|^2\\
&\times\,dy\,\dd x\,\dd\tau.
\end{aligned}
\label{eq:energy_ineq}
\end{align}
$\bullet$ \textbf{Step 4:} Letting $\eta\to 0$, notice that:
\begin{align*}
\begin{aligned}
&\delta_j u(x,\tau)\to\partial_{x_j}u(x,\tau)
\quad\text{a.e.\ and in }L^p_{\mathrm{loc}},\\
&\nabla^{(\sigma,\eta)}u(x)\to\nabla u(x,\tau)
\quad\text{a.e.\ for each }\sigma\in[0,1],\\
&\nabla^{(\sigma,\eta,y)}u(x)\to u(x,\tau)-u(y,\tau)
\quad\text{a.e.\ for each }\sigma\in[0,1].
\end{aligned}
\end{align*}
Since $\displaystyle\int_0^1|\nabla^{(\sigma,\eta)}u(x)|^{p-2}\dd\sigma
\to|\nabla u(x)|^{p-2}$ a.e. in $\R^N$ and $\displaystyle\int_0^1|\nabla^{(\sigma,\eta,y)}u(x)|^{p-2}\dd\sigma
\to|u(x)-u(y)|^{p-2}$ a.e. in $\R^N$, we get (1) by the weak lower semicontinuity of the $L^2$-norm applied to $\zeta\delta_jV$, $V:=|\nabla u|^{\frac{p-2}{2}}\nabla u$, since $|\delta_jV|^2\le\frac{p^2}{4}\bigl(\int_0^1|\nabla^{(\sigma,\eta)}u|^{p-2}\dd\sigma\bigr)|\nabla(\delta_ju)|^2$ and $|\nabla u|^{p-2}|\nabla(\partial_{x_j}u)|^2\le|\partial_{x_j}V|^2$, and (2) by Fatou's lemma:
\begin{align*}
\begin{aligned}
&1)\,\,\liminf_{\eta\to 0}\left[
\frac{1}{2}\int_{Q(\theta,\rho)}
\Bigl(\int_0^1|\nabla^{(\sigma,\eta)}u|^{p-2}\dd\sigma\Bigr)
|\nabla(\delta_j u)|^2\zeta^2\right]
\geq
\frac{2}{p^2}\int_{Q(\theta,\rho)}
|\nabla u|^{p-2}|\nabla(\partial_{x_j}u)|^2\zeta^2;\\
&2)\,\,\liminf_{\eta\to 0}\left[
\frac{p-1}{8}\int_{Q(\theta,\rho)}\!\int_{\R^N}
\frac{|\delta_j u(x)-\delta_j u(y)|^2
\displaystyle\int_0^1|\nabla^{(\sigma,\eta,y)}u|^{p-2}\dd\sigma}
{|x-y|^{N+sp}}
\frac{\zeta^2(x)+\zeta^2(y)}{2}\,\dd y\,\dd x\,\dd\tau\right]\\
&\qquad\geq
\frac{p-1}{8}\int_{Q(\theta,\rho)}\!\int_{\R^N}
\frac{|u(x)-u(y)|^{p-2}
|\partial_{x_j}u(x)-\partial_{x_j}u(y)|^2}
{|x-y|^{N+sp}}\frac{\zeta^2(x)+\zeta^2(y)}{2}\dd y\,\dd x\,\dd\tau,
\end{aligned}
\end{align*}
By convexity of $r\mapsto r^{p-2}$ ($p\geq 4$), check that:
\begin{equation*}
\begin{aligned}
&\bullet\,\,\displaystyle\int_0^1|\nabla^{(\sigma,\eta)}u|^{p-2}\dd\sigma
\leq C\bigl(|\nabla u(x)|^{p-2}+|\nabla u(x+\eta e_j)|^{p-2}\bigr),\\
&\bullet\,\,\displaystyle\int_0^1|\nabla^{(\sigma,\eta,y)}u|^{p-2}\dd\sigma
\leq C\bigl(|u(x)-u(y)|^{p-2}+|u(x+\eta e_j)-u(y+\eta e_j)|^{p-2}\bigr)
\end{aligned}
\end{equation*}
which is bounded in $L^{p/(p-2)}_{\mathrm{loc}}$ uniformly
in $\eta$. Consequently, the dominated convergence theorem gives
\begin{align*}
\begin{aligned}
&3)\,\,\lim_{\eta\to 0}
\int_{Q(\theta,\rho)}|\delta_j u|^2|\zeta_t|\zeta\,\dd x\,\dd\tau
= \int_{Q(\theta,\rho)}|\partial_{x_j}u|^2|\zeta_t|\zeta\,\dd x\,\dd\tau\\
&4)\,\,\lim_{\eta\to 0}
\int_{Q(\theta,\rho)}
\Bigl(\int_0^1|\nabla^{(\sigma,\eta)}u|^{p-2}\dd\sigma\Bigr)
|\delta_j u|^2|\nabla\zeta|^2\,\dd x\,\dd\tau
= \int_{Q(\theta,\rho)}
|\nabla u|^{p-2}|\partial_{x_j}u|^2|\nabla\zeta|^2\,\dd x\,\dd\tau\\
&5)\,\,\lim_{\eta\to 0}\int_{-\theta}^t\!\int_{K_\rho}\int_{\R^N}\frac{1}{|x-y|^{N+sp}}\left(\int_0^1\left|\nabla^{(\sigma,\eta,y)}u\right|^{p-2}\dd\sigma\right)\left|\delta_j u(x)+\delta_j u(y)\right|^2\left|\zeta(x)-\zeta(y)\right|^2\,dy\,\dd x\,\dd\tau\\
&\qquad=\int_{Q(\theta,\rho)}\!\int_{\R^N}
\frac{|u(x)-u(y)|^{p-2}}
{|x-y|^{N+sp}}\left|\partial_{x_j} u(x)+\partial_{x_j} u(y)\right|^2\left|\zeta(x)-\zeta(y)\right|^2\dd y\,\dd x\,\dd\tau.
\end{aligned}
\end{align*}
\noindent
Passing to the limit in \eqref{eq:energy_ineq} and using 1) - 5), we deduce
\begin{align*}
\begin{aligned}
&\int_{Q(\theta,\rho)}
|\nabla u|^{p-2}|\nabla(\partial_{x_j}u)|^2\zeta^2\,\dd x\,\dd\tau\\
&\leq \gamma\left[\int_{Q(\theta,\rho)}
\Bigl(|\nabla u|^2\,|\zeta_t|\zeta+|\nabla u|^p|\nabla\zeta|^2\Bigr)\,\dd x\,\dd\tau\right.\\
&\left.+2\int_{Q(\theta,\rho)}\!\int_{\R^N}
\frac{|u(x)-u(y)|^{p-2}}
{|x-y|^{N+sp}}\left(\left|\partial_{x_j}u(x)\right|^2+\left|\partial_{x_j}u(y)\right|^2\right)\left|\zeta(x)-\zeta(y)\right|^2\dd y\,\dd x\,\dd\tau\right].
\end{aligned}
\end{align*}
Choosing $\zeta\equiv 1$ on $Q(\sigma\theta,\sigma\rho)$,
$\zeta=0$ outside $Q(\theta,\rho)$, with:
\begin{equation}
|\nabla\zeta|\leq\frac{1}{(1-\sigma)\rho},
\qquad
0\leq\zeta_t\leq\frac{1}{(1-\sigma)\theta},
\label{eq:cutoff}
\end{equation}
we conclude
\begin{align*}
\begin{aligned}
&\int_{Q(\theta,\rho)}
|\nabla u|^{p-2}|\nabla(\partial_{x_j}u)|^2\zeta^2\,\dd x\,\dd\tau\\
&\leq \gamma\left[\int_{Q(\theta,\rho)}
\Bigl(\frac{|\nabla u|^2}{(1-\sigma)\theta}+\frac{|\nabla u|^p}{(1-\sigma)^2\rho^2}
\Bigr)\,\dd x\,\dd\tau\right.\\
&\left.+\int_{Q(\theta,\rho)}\!\int_{\R^N}
\frac{|u(x)-u(y)|^{p-2}
}
{|x-y|^{N+sp}}\left(\left|\partial_{x_j}u(x)\right|^2+\left|\partial_{x_j}u(y)\right|^2\right)\left|\zeta(x)-\zeta(y)\right|^2\dd y\,\dd x\,\dd\tau\right].
\end{aligned}
\end{align*}
This completes the proof.
\end{proof}
\begin{remark}
Notice that if $u\in L^{p}(0,T;W^{1,p}(\R^N))\bigcap L^{p}(0,T;W^{s,p}(\R^N))$, then
\begin{align*}
\iint_{Q(\theta,\rho)}\!\int_{\R^N}
\frac{|u(x,\tau)-u(y,\tau)|^{p-2}
\bigl(|\partial_{x_j}u(x,\tau)|^2+|\partial_{x_j}u(y,\tau)|^2\bigr)\left|\zeta(x)-\zeta(y)\right|^2}
{|x-y|^{N+sp}}\dd y\,\dd x\,\dd\tau<\infty.
\end{align*}
Indeed, from Hölder's inequalities and Fubini's theorem, it holds
\begin{align}\label{DiBene_fractional}
\begin{aligned}
&\int_{Q(\theta,\rho)}\!\int_{\R^N}
\left[\frac{|u(x)-u(y)|}{|x-y|^{\frac{N+sp}{p}}}\right]^{p-2}
\!\!\bigl(|\partial_{x_j}u(x)|^2+|\partial_{x_j}u(y)|^2\bigr)
\left[\frac{|\zeta(x)-\zeta(y)|}{|x-y|^{\frac{N+sp}{p}}}\right]^2
\dd y\,\dd x\,\dd\tau\\
&=\int_{Q(\theta,\rho)}\!|\partial_{x_j}u(x)|^2
\int_{\R^N}
\left[\frac{|u(x)-u(y)|}{|x-y|^{\frac{N+sp}{p}}}\right]^{p-2}
\!\!\left[\frac{|\zeta(x)-\zeta(y)|}{|x-y|^{\frac{N+sp}{p}}}\right]^2
\dd y\,\dd x\,\dd\tau\\
&\quad+\int_{Q(\theta,\rho)}\!\int_{\R^N}
\left[\frac{|u(x)-u(y)|}{|x-y|^{\frac{N+sp}{p}}}\right]^{p-2}
\!\!|\partial_{x_j}u(y)|^2
\left[\frac{|\zeta(x)-\zeta(y)|}{|x-y|^{\frac{N+sp}{p}}}\right]^2
\dd y\,\dd x\,\dd\tau\\
&\leq\int_{Q(\theta,\rho)}\!|\partial_{x_j}u(x)|^2
\left(\int_{\R^N}
\frac{|u(x)-u(y)|^p}{|x-y|^{N+sp}}\dd y\right)^{\frac{p-2}{p}}
\!\!\left(\int_{\R^N}
\frac{|\zeta(x)-\zeta(y)|^p}{|x-y|^{N+sp}}\dd y\right)^{\frac{2}{p}}
\dd x\,\dd\tau\\
&\quad+\int_{Q(\theta,\rho)}
\left(\int_{\R^N}
\frac{|u(x)-u(y)|^p}{|x-y|^{N+sp}}\dd y\right)^{\frac{p-2}{p}}
\!\!\left(\int_{\R^N}
\frac{|\partial_{x_j}u(y)|^p|\zeta(x)-\zeta(y)|^p}
{|x-y|^{N+sp}}\dd y\right)^{\frac{2}{p}}
\dd x\,\dd\tau\\
&\leq\left(\int_{Q(\theta,\rho)}\!|\partial_{x_j}u(x)|^p\dd x\,\dd\tau\right)^{\frac{2}{p}}
\left(\int_{Q(\theta,\rho)}\!\int_{\R^N}
\frac{|u(x)-u(y)|^p}{|x-y|^{N+sp}}\dd y\!\left[\int_{\R^N}
\frac{|\zeta(x)-\zeta(y)|^p}{|x-y|^{N+sp}}\dd y\right]^{\frac{2p}{p(p-2)}}
\dd x\,\dd\tau\right)^{\frac{p-2}{p}}\\
&\quad+\left(\int_{Q(\theta,\rho)}
\int_{\R^N}
\frac{|u(x)-u(y)|^p}{|x-y|^{N+sp}}\dd y\,\dd x\,\dd\tau\right)^{\frac{p-2}{p}}
\!\!\left(\int_{0}^T\int_{\R^N}|\partial_{x_j}u(y)|^p\left[\int_{K_{\rho}}
\frac{|\zeta(x)-\zeta(y)|^p}
{|x-y|^{N+sp}}\dd x\right]\,\dd y\,\dd\tau\right)^{\frac{2}{p}}
\end{aligned}
\end{align}
Notice that since $\zeta\in C^{0,\alpha}$ with $\alpha>s$, then
\begin{align*}
\int_{\R^N}
\frac{|\zeta(x)-\zeta(y)|^p}{|x-y|^{N+sp}}\dd y&\le C^{p}\int_{|x-y|\le1}\frac{|x-y|^{\alpha p}}{|x-y|^{N+sp}}\dd y+2^{p}\|\zeta\|_{\infty}^p\int_{|x-y|>1}\frac{1}{|x-y|^{N+sp}}\dd y\\
&\le C^{p}\int_{|z|\le1}|z|^{\alpha p-N-sp}\dd z+2^{p}\|\zeta\|_{\infty}^p\int_{|z|>1}\frac{1}{|z|^{N+sp}}\dd z\\
&\le C^{p}|\mathbb{S}^{N-1}|\int_{0}^1r^{-1-sp+\alpha p}\dd r+2^{p}\|\zeta\|_{\infty}^p|\mathbb{S}^{N-1}|\int_{1}^{\infty}r^{-1-sp}\dd r\\
&\le |\mathbb{S}^{N-1}|\left(C^{p}\frac{1}{p(\alpha-s)}+2^{p}\|\zeta\|_{\infty}^p\frac{1}{sp}\right):=C_{\zeta}
\end{align*}
Consequently from \eqref{DiBene_fractional} we deduce
\begin{align*}
&\int_{Q(\theta,\rho)}\!\int_{\R^N}
\left[\frac{|u(x)-u(y)|}{|x-y|^{\frac{N+sp}{p}}}\right]^{p-2}
\!\!\bigl(|\partial_{x_j}u(x)|^2+|\partial_{x_j}u(y)|^2\bigr)
\left[\frac{|\zeta(x)-\zeta(y)|}{|x-y|^{\frac{N+sp}{p}}}\right]^2
\dd y\,\dd x\,\dd\tau\\
&\le 2C_{\zeta}^{\frac{2}{p}}\left\|u\right\|_{L^{p}(0,T;W^{s,p}(\R^N))}^{p-2}
\|u\|_{L^{p}(0,T;W^{1,p}(\R^N))}^2<\infty.
\end{align*}
\end{remark}

\begin{corollary}[Global weighted second-order estimate]
\label{cor:global_DiBe}
Let $p\geq4$, $s\in(0,1)$ with $1<sp<2$, $N\geq1$, and let $u$ be the solution of \eqref{eq:PDE}
given by Theorem~\ref{thm:existence_PDE}. Then, for every $T>0$,
\begin{equation*}
\int_0^T\!\int_{\R^N}|\nabla u|^{p-2}|D^2u|^2\,\dd x\,\dd t
\le\gamma\,\|\nabla u_0\|_{L^2(\R^N)}^2<\infty,
\end{equation*}
with $\gamma=\gamma(N,p)$. In particular $C_0:=\int_0^T\!\int_{\R^N}|\nabla u|^{p-2}|D^2u|^2\,\dd x\,\dd t<\infty$.
\end{corollary}

\begin{proof} \mbox{} \\[1ex]
\textbf{Step 1: Cylinder down to $t=0$.} 

\medskip

\noindent We repeat the proof of Proposition~\ref{prop:hybrid31} on the cylinder $K \times (0,T)$ with the following setup:
\begin{itemize}
    \setlength{\itemsep}{1ex}
    \item[$\bullet$] $K$ is an open cube,
    \item[$\bullet$] $\zeta \in C_c^\infty(K)$ is a time-independent cutoff function ($\zeta_t \equiv 0$).
\end{itemize}

\smallskip

\noindent Since $\zeta_t \equiv 0$, the terms in \eqref{eq:time} evaluate as follows:
\[
\begin{cases}
    \displaystyle \int |\delta_j u_h|^2 \zeta \zeta_t = 0 & \text{(term with } \zeta_t), \\[3ex]
    \displaystyle \frac{1}{2} \int_K \bigl| [\delta_j u]_h(0) \bigr|^2 \zeta^2 \, \dd x \neq 0 \quad & \text{(boundary term at } \tau = 0\text{, which no longer vanishes)}.
\end{cases}
\]

\noindent The limit of this expression relies on two convergences:
\begin{itemize}
    \setlength{\itemsep}{2.5ex}
    
    \item[$\bullet$] $[\delta_j u]_h(0) \to \delta_j u_0$ in $L^2(K)$ \quad (since $u \in C\bigl([0,T]; L^2(\mathbb{R}^N)\bigr)$).
    
    \item[$\bullet$] $\delta_j u_0 \to \partial_{x_j} u_0$ in $L^2(\mathbb{R}^N)$ \quad (since $\nabla u_0 \in L^2(\mathbb{R}^N)$ by assumptions (A5)--(A6) on the initial datum: $\nabla u_0 \in L^p$ is supported in $B_{R_0}$ and $p \ge 2$).
\end{itemize}
\smallskip

\noindent This directly yields the limit:
\[
\lim_{\eta\to0} \int_K |\delta_j u_0|^2 \zeta^2 \, \dd x 
=
\int_{\mathbb{R}^N} |\partial_{x_j} u_0|^2 \zeta^2 \, \dd x.
\]
\noindent Passing to the limits $h \to 0$ and then $\eta \to 0$ exactly as in Steps 1--4 of the proof of Proposition~\ref{prop:hybrid31}, and using the pointwise bound $|\nabla u|^{p-2}|\partial_{x_j}u|^2 \le |\nabla u|^p$, we obtain for every $j$:
\begin{align}
\int_0^T \!\! \int_{\mathbb{R}^N} & |\nabla u|^{p-2} |\nabla(\partial_{x_j} u)|^2 \zeta^2 \, \dd x \, \dd\tau \nonumber \\
&\le \gamma \Biggl[ 
\int_{\mathbb{R}^N} |\partial_{x_j} u_0|^2 \zeta^2 \, \dd x 
+ \int_0^T \!\! \int_{\mathbb{R}^N} |\nabla u|^p |\nabla \zeta|^2 \, \dd x \, \dd\tau \nonumber \\
&\quad + \int_0^T \!\! \iint_{\mathbb{R}^N \times \mathbb{R}^N} 
\frac{|u(x) - u(y)|^{p-2}}{|x - y|^{N+sp}} 
\Bigl( |\partial_{x_j} u(x)|^2 + |\partial_{x_j} u(y)|^2 \Bigr) |\zeta(x) - \zeta(y)|^2 \, \dd y \, \dd x \, \dd\tau 
\Biggr].
\label{eq:cyl0}
\end{align}

\medskip

\noindent\textbf{Step 2: Dilation of the cutoff.} 

\medskip

\noindent We fix a radial nonincreasing function $\zeta \in C_c^\infty(B_2)$ such that $0 \le \zeta \le 1$ and $\zeta \equiv 1$ on $B_1$. For $R \ge 1$, we define the rescaled cutoff:
\[
\zeta_R(x) := \zeta\left(\frac{x}{R}\right).
\]

\noindent Applying \eqref{eq:cyl0} with $\zeta = \zeta_R$ on the cube $K = (-2R,2R)^N$, we get:
\[
\int_0^T \!\! \int_{\mathbb{R}^N} |\nabla u|^{p-2} |\nabla(\partial_{x_j}u)|^2 \zeta_R^2 \, \dd x \, \dd\tau
\le \gamma \Bigl[ \|\partial_{x_j}u_0\|_{L^2(\mathbb{R}^N)}^2 + I_R + J_R \Bigr].
\]

\noindent The error terms $I_R$ and $J_R$ are bounded as follows:
\begin{align*}
I_R &:= \int_0^T \!\! \int_{\mathbb{R}^N} |\nabla u|^p |\nabla\zeta_R|^2 \, \dd x \, \dd\tau 
\le \frac{\|\nabla\zeta\|_\infty^2}{R^2} \, \|\nabla u\|_{L^p((0,T)\times\mathbb{R}^N)}^p, \\[3ex]
J_R &:= \int_0^T \!\! \iint_{\mathbb{R}^N\times\mathbb{R}^N}
\frac{|u(x)-u(y)|^{p-2}}{|x-y|^{N+sp}}
\bigl(|\partial_{x_j}u(x)|^2+|\partial_{x_j}u(y)|^2\bigr) |\zeta_R(x)-\zeta_R(y)|^2 \, \dd y \, \dd x \, \dd\tau.
\end{align*}

\medskip

\noindent\textbf{Step 3: Letting $R \to \infty$.} 

\medskip

\noindent We analyze the asymptotic behavior of each term:

\begin{itemize}
    \setlength{\itemsep}{2.5ex}

    \item[$\bullet$] \textbf{Limit of $I_R$:} Clearly, $I_R \xrightarrow[R \to \infty]{} 0$.

    \item[$\bullet$] \textbf{Domination of $J_R$:} Using the Lipschitz bound $|\zeta_R(x)-\zeta_R(y)| \le \min\bigl\{1, \|\nabla\zeta\|_\infty \frac{|x-y|}{R}\bigr\}$, the integrand of $J_R$ tends to $0$ pointwise. It is bounded uniformly in $R \ge 1$ by the function $G$:
    \[
    G(\tau,x,y) := \max\{1, \|\nabla\zeta\|_\infty^2\} \frac{|u(x)-u(y)|^{p-2}}{|x-y|^{N+sp}} \bigl(|\partial_{x_j}u(x)|^2+|\partial_{x_j}u(y)|^2\bigr) \min\{1, |x-y|^2\}.
    \]

    \item[$\bullet$] \textbf{Integrability of $G$:} By symmetry and H\"older's inequality (exponents $\dfrac{p}{p-2}$ and $\dfrac{p}{2}$):
    \begin{align*}
    \int_0^T \!\! \iint_{\mathbb{R}^N\times\mathbb{R}^N} G 
    &\le 2\max\{1, \|\nabla\zeta\|_\infty^2\} \Biggl( \int_0^T \!\! \iint_{\mathbb{R}^N\times\mathbb{R}^N} \frac{|u(x)-u(y)|^p}{|x-y|^{N+sp}}\,\dd x\,\dd y\,\dd t \Biggr)^{\frac{p-2}{p}} \\
    &\quad \times \Biggl( \int_0^T \!\! \int_{\mathbb{R}^N} |\nabla u(x)|^p \int_{\mathbb{R}^N} \frac{\min\{1, |x-y|^p\}}{|x-y|^{N+sp}}\,\dd y\,\dd x\,\dd t \Biggr)^{\frac{2}{p}} \\[2ex]
    &\le 2\max\{1, \|\nabla\zeta\|_\infty^2\} \Biggl( \omega_N \Bigl(\frac{1}{p-sp} + \frac{1}{sp}\Bigr) \Biggr)^{\frac{2}{p}} \|u\|_{L^p(0,T;W^{s,p}(\mathbb{R}^N))}^{p-2} \, \|\nabla u\|_{L^p((0,T)\times\mathbb{R}^N)}^2 \\[1ex]
    &< \infty \quad \text{(by Theorem~\ref{thm:existence_PDE})}.
    \end{align*}
    By the dominated convergence theorem, $J_R \xrightarrow[R \to \infty]{} 0$.

    \item[$\bullet$] \textbf{Conclusion:} Since $\zeta_R \uparrow 1$, the monotone convergence theorem applied to the left-hand side yields:
    \[
    \int_0^T \!\! \int_{\mathbb{R}^N} |\nabla u|^{p-2} |\nabla(\partial_{x_j}u)|^2 \, \dd x \, \dd\tau \le \gamma \, \|\partial_{x_j}u_0\|_{L^2(\mathbb{R}^N)}^2.
    \]
\end{itemize}

\noindent Summing this final inequality over $j=1,\ldots,N$ concludes the proof.
\end{proof}

%\bibliographystyle{plain}
%\bibliography{main}
\end{document}